\documentclass[12pt,a4paper,oneside,english]{amsart}
\usepackage[T1]{fontenc}
\usepackage[latin9]{inputenc}
\usepackage{babel}
\usepackage{varioref}
\usepackage{amsthm}
\usepackage{amsbsy}
\usepackage{amstext}
\usepackage{amssymb}
\usepackage{esint}
\usepackage[unicode=true,pdfusetitle,
 bookmarks=true,bookmarksnumbered=false,bookmarksopen=false,
 breaklinks=false,pdfborder={0 0 1},backref=false,colorlinks=false]
 {hyperref}
\hypersetup{
 a4paper}
\usepackage{breakurl}

\makeatletter

\newcommand{\noun}[1]{\textsc{#1}}

\numberwithin{equation}{section}
\numberwithin{figure}{section}

\usepackage{amsthm}\usepackage{amsbsy}\usepackage{amstext}
\usepackage[a4paper,margin=2cm]{geometry}

\numberwithin{equation}{section}
\numberwithin{figure}{section}
\theoremstyle{plain}
\newtheorem{thm}{\protect\theoremname}
  \theoremstyle{definition}
  
  \theoremstyle{plain}
  \newtheorem{assumption}[thm]{\protect\assumptionname}
  \theoremstyle{plain}
  
  \theoremstyle{plain}

\providecommand{\assumptionname}{Assumption}
  \providecommand{\definitionname}{Definition}
  \providecommand{\lemmaname}{Lemma}
  \providecommand{\propositionname}{Proposition}
\providecommand{\theoremname}{Theorem}

\newtheorem{theorem}{Theorem}
\newtheorem{acknowledgement}[theorem]{Acknowledgement}

\newtheorem{corollary}[theorem]{Corollary}

\newtheorem{definition}[theorem]{Definition}
\newtheorem{example}[theorem]{Example}

\newtheorem{lemma}[theorem]{Lemma}

\newtheorem{proposition}[theorem]{Proposition}
\newtheorem{remark}[theorem]{Remark}

\input{tcilatex}

\providecommand{\assumptionname}{Assumption}
  \providecommand{\definitionname}{Definition}
  \providecommand{\lemmaname}{Lemma}
  \providecommand{\propositionname}{Proposition}
\providecommand{\theoremname}{Theorem}

\makeatother

\begin{document}

\author{Peter Friz}

\address{PF is affiliated with WIAS, Mohrenstrasse 39, 10117 Berlin and TU
Berlin, Institut für Mathematik, Strasse des 17. Juni 136, 10623 Berlin.}

\author{Harald Oberhauser}

\address{HO is corresponding author and affiliated with TU Berlin, Institut
für Mathematik, Strasse des 17. Juni 136, 10623 Berlin. }

\email{h.oberhauser@gmail.com}

\title{Rough path stability of (semi-)linear SPDEs}
\begin{abstract}
We give meaning to linear and semi-linear (possibly degenerate) parabolic
partial differential equations with (affine) linear rough path noise
and establish stability in a rough path metric. In the case of enhanced
Brownian motion (Brownian motion with its Lévy area) as rough path
noise the solution coincides with the standard variational solution
of the SPDE.
\end{abstract}
\maketitle

\section{Introduction}

Given a continuous, $d$-dimensional semimartingale $Z=\left(Z^{1},\ldots,Z^{d}\right)$
consider the SPDE
\begin{equation}
du+L\left(t,x,u,Du,D^{2}u\right)dt=\sum_{k=1}^{d}\Lambda_{k}\left(t,x,u,Du\right)\circ dZ_{t}^{k},\label{SPDE}
\end{equation}
with scalar initial data $u\left(0,\cdot\right)=u_{0}\left(\cdot\right)$
on $\mathbb{R}^{n},$ $L$ a (semi-)linear second order operator of
the form
\[
L\left(t,x,r,p,X\right)=-\text{Tr}\left[A\left(t,x\right)\cdot X\right]+b\left(t,x\right)\cdot p+c\left(t,x,r\right)
\]
and $\Lambda$ a collection of first order different operators $\Lambda_{k}=\Lambda_{k}\left(t,x,r,p\right)$
which are affine linear in $r,p$, that is,
\begin{equation}
\Lambda_{k}\left(t,x,r,p\right)=p\cdot\sigma_{k}\left(t,x\right)+r\,\nu_{k}\left(t,x\right)+g_{k}\left(t,x\right),\ \ \ k=1,\dots,d.\label{eq:lambda}
\end{equation}
The contribution of this article is to give meaning to equation (\ref{SPDE})
when $Z\left(\omega\right)$ is replaced by a rough path $\mathbf{z}$
(this is carried out in sections \ref{sec:backg},\ref{sec:trans}
and \ref{sec:rpde}). Our main result as stated and proven in section
\ref{sec:rpde} (in section \ref{sec:backg} we recall $Lip^{\gamma}$-regularity,
rough paths and their metrics and the $BUC$ space of bounded, uniformly
continuous real-valued functions that appear in the theorem below)
is the following

\begin{theorem}Let $p\geq1$. Assume $L$ fulfills assumption \ref{main_assumptions}
and the coefficients of $\Lambda=\left(\Lambda_{1},\dots,\Lambda_{d_{1}+d_{2}+d_{3}}\right)$
fulfill assumption \ref{as:lambda} for some $\gamma>p+2$ (assumption
\ref{main_assumptions}\&\ref{as:lambda} are given in section \vref{sec:rpde}).
Let $u_{0}\in BUC\left(\mathbb{R}^{n}\right)$ and let $\mathbf{z}$
be a geometric $p$-rough path. Then there exists a unique $u=u^{\mathbf{z}}\in BUC\left(\left[0,T\right]\times\mathbb{R}^{n}\right)$
such that for any sequence $\left(z^{\epsilon}\right)_{\epsilon}\subset C^{1}\left(\left[0,T\right],\mathbb{R}^{d}\right)$
such that $z^{\varepsilon}\rightarrow\mathbf{z}$ in $p$-rough path
sense, the viscosity solutions $\left(u^{\varepsilon}\right)\subset BUC\left(\left[0,T\right]\times\mathbb{R}^{n}\right)$
of
\[
\dot{u}^{\varepsilon}+L\left(t,x,u^{\varepsilon},Du^{\varepsilon},D^{2}u^{\varepsilon}\right)=\sum_{k=1}^{d}\Lambda_{k}\left(t,x,u^{\varepsilon},Du^{\varepsilon}\right)\dot{z}_{t}^{k;\varepsilon}\mathbf{,\,\,\,}u^{\epsilon}\left(0,\cdot\right)=u_{0}\left(\cdot\right),
\]
converge locally uniformly to $u^{\mathbf{z}}$. We write formally,
\[
du+L\left(t,x,u,Du,D^{2}u\right)dt=\Lambda\left(t,x,u,Du\right)d\mathbf{z}_{t},\,\,\, u\left(0,\cdot\right)=u_{0}\left(\cdot\right).
\]
Moreover, we have the contraction property
\[
\sup_{\left(t,x\right)\in\mathbb{R}^{n}\times\left[0,T\right]}\left|u^{\mathbf{z}}\left(t,x\right)-\hat{u}^{\mathbf{z}}\left(t,x\right)\right|\leq e^{CT}\sup_{x\in\mathbb{R}^{n}}\left\vert u_{0}\left(x\right)-\hat{u}_{0}\left(x\right)\right\vert 
\]
($C$ given by \eqref{EqWeakProper}) and continuity of the solution
map $\left(\mathbf{z,}u_{0}\right)\mapsto u^{\mathbf{z}}$ 
\[
C^{0,p\text{-var}}\left(\left[0,T\right],G^{\left[p\right]}\left(\mathbb{R}^{d}\right)\right)\times\mathrm{BUC}\left(\mathbb{R}^{n}\right)\rightarrow\mathrm{BUC}\left(\left[0,T\right]\times\mathbb{R}^{n}\right).
\]
\end{theorem}

The resulting theory of rough PDEs can then be used (in a ``rough-pathwise''
fashion) to give meaning (and then existence, uniqueness, stability,
etc.) to large classes of stochastic partial differential equations
which has numerous benefits as discussed in section \ref{AppSPDE}.
By combining well-known Wong--Zakai type results of the $L^{2}$-theory
of SPDEs \cite{MR2268661,MR1313027,MR1390565,MR2267702} with convergence
of piecewise linear approximations to ``enhanced'' Brownian motion
(EBM) in rough path sense, e.g.~\cite[Chapter 13 and 14]{friz-victoir-book},
we show that the solutions provided by above theorem when applied
with EBM as rough path are in fact the usual $L^{2}$-solutions of
the variational approach, \cite{MR557763,MR553909,MR1135324}. This
``intersection'' of RPDE/SPDE theory is made precise in Section
\ref{sec:viscL2}. However, let us emphasize that neither theory is
``contained'' in the other, even in the case of Brownian driving
noise. An appealing feature of our RPDE approach is that it can handle
degenerate situations (including pure first order SPDEs) and automatically
yields continuous versions of SPDE solutions without requiring dimension-dependent
regularity assumptions on the coefficients (as pointed out by Krylov
\cite{MR1661766}, a disadvantage of the $L^{2}$ theory of SPDEs).
On the other hand, our regularity assumption (in particular in the
noise terms) are more stringent than what is needed to ensure existence
and uniqueness in the $L^{2}$ theory of SPDEs. Below we sketch our
approach and the outline of this article.

\subsection{Robustification}

In fact, it is part of folklore that the equation \eqref{SPDE} can
be given a pathwise meaning in the case when there is no gradient
noise ($\sigma=0$ in \eqref{eq:lambda}).

\textbf{Classical robustification:} if $\sigma=0$ in \eqref{eq:lambda}
and also (for simplicity of presentation only) $\nu=\nu\left(x\right)$
(i.e.~no time dependence) one can take a smooth path $z$ and solve
the auxiliary differential equation $\dot{\phi}=$ $\phi\,\sum_{j}\nu_{j}\left(x\right)dz^{j}\equiv\phi\,\nu\cdot dz$.
The solution is given by
\[
\phi_{t}=\phi_{0}\exp\left(\int_{0}^{t}\nu\left(x\right)\cdot dz\right)=\phi_{0}\exp\left(\nu\left(x\right)\cdot z_{t}\right)
\]
and induces the flow map $\phi\left(t,\phi_{0}\right):=\phi_{0}\exp\left(\nu\left(x\right)\cdot z_{t}\right)$;
observe that these expressions can be extended by continuity to any
continuous path $z$ such as a typical realization of $Z_{\cdot}\left(\omega\right)$.
The point is that this transform allows to transform the SPDE\ into
a random PDE (sometimes called the Zakai equation in robust form):
it suffices to introduce $v$ via the ``\textbf{outer transform}''$u\left(t,x\right)=\phi\left(t,v\left(t,x\right)\right)$
which leads immediately to 
\[
v\left(t,x\right)=\exp\left(-\nu\left(x\right)\cdot z_{t}\right)u\left(t,x\right).
\]
An elementary computation then shows that $v$ solves a linear PDE
given by an affine linear operator $^{\phi}L$ in $v,Dv,D^{2}v$ with
coefficients that will depend on $z$ resp. $Z_{t}\left(\omega\right)$,
\[
dv+\,^{\phi}L\left(t,x,v,Dv,D^{2}v\right)dt=0.
\]
Moreover, one can conclude from this representation that $u=u\left(z\right)$
is continuous with respect to the uniform metric $\left\vert z-\widetilde{z}\right\vert _{\infty;\left[0,T\right]}=\sup_{r\in\left[0,T\right]}\left|z_{r}-\tilde{z}_{r}\right|$.
This provides a fully pathwise ``robust'' approach (the extension
to vector field $\nu=\nu\left(t,x\right)$ with sufficiently smooth
time-dependence is easy).

\textbf{Rough path robustification:} \ The classical robustification
\emph{does }\textit{not}\emph{ work }in presence of general gradient
noise. In fact, we can not expect PDE solutions to 
\[
du+L\left(t,x,u,Du,D^{2}u\right)dt=\sum_{k=1}^{d}\Lambda_{k}\left(t,x,u,Du\right)\, dz_{t}^{k},
\]
(which are well-defined for smooth $z:\left[0,T\right]\rightarrow\mathbb{R}^{d}$)
to depend continuously on $z$ in uniform topology (cf.~the ``twisted
approximations''\ of section \ref{AppSPDE}). Our main result is
that $u=u(z)$ is continuous with respect to rough path metric%
\footnote{Two (smooth) paths $z,\tilde{z}$ are close in rough path metric iff
$z$ is close to $\tilde{z}$ AND\ sufficiently many iterated integrals
of $z$ are close to those of $\tilde{z}$. More details are given
later in this article as needed.%
}. That is, if $\left(z^{n}\right)\subset C^{1}\left(\left[0,T\right],\mathbb{R}^{d}\right)$
is Cauchy in rough path metric then $\left(u^{n}\right)$ will converge
to a limit which will be seen to depend only on the (rough path) limit
of $z^{n}$ (and not on the approximating sequence). As a consequence,
it is meaningful to replace $z$ above by an abstract (geometric)
rough path $\mathbf{z}$ and the analogue of\ Lyons' universal limit
theorem \cite{lyons-qian-02} holds.

\subsection{Structure and outline}

We shall prefer to write the right hand side of (\ref{SPDE}) in the
equivalent form
\[
\sum_{i=1}^{d_{1}}\left(Du\cdot\sigma_{i}\left(t,x\right)\right)\circ dZ_{t}^{1;i}+u\sum_{j=1}^{d_{2}}\nu_{j}\left(t,x\right)\circ dZ^{2;j}+\sum_{k=1}^{d_{3}}g_{k}\left(t,x\right)\circ dZ^{3;k}
\]
where $Z\equiv(Z^{1},Z^{2},Z^{3})$ and $Z^{i}$ is a $d_{i}$-dimensional,
continuous semimartingale. Our approach is based on a pointwise (viscosity)
interpretation of (\ref{SPDE}): we successively transform away the
noise terms such as to transform the SPDE, ultimately,\ into a random
PDE. The big scheme of the paper is
\begin{eqnarray*}
 &  & u\,\,\overset{\text{Transformation 1}}{\mapsto}\, u^{1}\text{ where }u^{1}\text{ has the (gradient) noise driven by }Z^{1}\text{ removed;}\\
 &  & u^{1}\overset{\text{Transformation 2}}{\mapsto}u^{12}\text{ where }u^{12}\text{ has the remaining noise driven by }Z^{2}\text{ removed;}\\
 &  & u^{12}\overset{\text{Transformation 3}}{\mapsto}\tilde{u}\text{ where }\tilde{u}\text{ has the remaining noise driven by }Z^{3}\text{ removed.}
\end{eqnarray*}
None of these transformations is new on its own. The first is an example
of Kunita's stochastic characteristics method; the second is known
as robustification (also know as Doss--Sussman transform); the third
amounts to change $u^{12}$ additively by a random amount and has
been used in virtually every SPDE context with additive noise.%
\footnote{Transformation 2 and 3 could actually be performed in 1 step; however,
the separation leads to a simpler analytic tractability of the transformed
equations.%
} What is new is that the combined transformation can be managed and
is compatible with rough path convergence; for this we have to remove
all probability from the problem: In fact, we will transform an RPDE
(rough PDE) solution $u$ into a classical PDE solution $\tilde{u}$
in which the coefficients depend on various rough flows (i.e.\ the
solution flows to rough differential equations) and their derivatives.
Stability results of rough path theory and viscosity theory, in the
spirit of \cite{MR2765508,caruana-2008}, then play together to yield
the desired result. Upon using the canonical rough path lift of the
observation process in this RPDE one has constructed a robust version
of the SPDE solution of equation (\ref{SPDE}). We note that the viscosity/Stratonovich
approach allows us to \emph{avoid any ellipticity assumption on }$L;$
we can even handle the fully degenerate first order case. In turn,
we only obtain $BUC$\ (bounded, uniformly continuous) solutions.
Stronger assumptions would allow to discuss all this in a classical
context (i.e.~$\tilde{u}$ would be a $C^{1,2}$ solution) and SPDE
solution can then be seen to have certain spatial regularity, etc.

We should remark that the usual way to deal with (\ref{SPDE}), which
goes back to Pardoux, Krylov, Rozovskii, and others, \cite{MR2268661,MR1313027,MR1390565,MR2267702},
is to find solutions in a suitable functional analytic setting; e.g.\ such
that solutions evolve in suitable Sobolev spaces. The equivalence
of this solution concept with the RPDE approach as presented in sections
\ref{sec:backg} to \ref{sec:rpde} is then discussed in section \ref{sec:viscL2}.
Interestingly, there has been no success until now (despite the advances
by Deya--Gubinelli--Tindel \cite{gubinelli-tindel-2008,deya2011non}
and Teichmann \cite{TeichmannRPDE}) to include (\ref{SPDE}) in a
setting of abstract \textit{rough} evolution equations on infinite-dimensional
spaces.

\begin{acknowledgement}The first author would like to thank the organizers
and participants of the Filtering Workshop in June 2010, part of the
Newton Institute's SPDE program, where parts of this work was first
presented. The second author would like to thank the organizers and
participants of the Rough path and SPDE workshop, part of the Newton
Institute's SPDE program and expresses his gratitude for a Newton
Institute Junior membership grant. Partial support from the European
Unions Seventh Framework Programme (FP7/2007--2013)/ERC grant agreement
nr.~258237 is gratefully acknowledged. HO is indebted to G.~Barles,
P. Souganidis and the participants of the C.I.M.E.~meeting on HJB-equations
in 2011 for helpful conversations. Both authors would like to thank
the referees for their valuable comments.\end{acknowledgement}

\section{Background on viscosity theory and rough paths}

\label{sec:backg}

Let us recall some basic ideas of (second order) viscosity theory
\cite{MR1118699UserGuide,MR2179357FS} and rough path theory \cite{lyons-qian-02,MR2314753}.\ As
for viscosity theory, consider a real-valued function $u=u\left(t,x\right)$
with $t\in\left[0,T\right],x\in\mathbb{R}^{n}$ and assume $u\in C^{2}$
is a classical subsolution,
\[
\partial_{t}u+F\left(t,x,u,Du,D^{2}u\right)\leq0,
\]
where $F$ is a (continuous) function, \textit{degenerate elliptic}
in the sense that 
\[
F\left(t,x,r,p,A+B\right)\leq F\left(t,x,r,p,A\right)
\]
 whenever $B\geq0$ in the sense of symmetric \ matrices (cf.~\cite{MR1118699UserGuide}).
The idea is to consider a (smooth) test function $\varphi$ and look
at a local maxima $\left(\hat{t},\hat{x}\right)$ of $u-\varphi$.
Basic calculus implies that $Du\left(\hat{t},\hat{x}\right)=D\varphi\left(\hat{t},\hat{x}\right),\, D^{2}u\left(\hat{t},\hat{x}\right)\leq D\varphi\left(\hat{t},\hat{x}\right)$
and, from degenerate ellipticity,
\begin{equation}
\partial_{t}\varphi+F\left(\hat{t},\hat{x},u,D\varphi,D^{2}\varphi\right)\leq0.\label{Gxbar}
\end{equation}
This suggests to define a \textit{viscosity supersolution} (at the
point $\left(\hat{x},\hat{t}\right)$) to $\partial_{t}+F=0$ as a
continuous function $u$ with the property that (\ref{Gxbar}) holds
for any test function. Similarly, \textit{viscosity subsolutions}
are defined by reversing inequality in (\ref{Gxbar}); \textit{viscosity
solutions} are both super- and subsolutions. A different point of
view is to note that $u\left(t,x\right)\leq u\left(\hat{t},\hat{x}\right)-\varphi\left(\hat{t},\hat{x}\right)+\varphi\left(t,x\right)$
for $\left(t,x\right)$ near $\left(\hat{t},\hat{x}\right)$. A simple
Taylor expansion then implies
\begin{equation}
u\left(t,x\right)\leq u\left(\hat{t},\hat{x}\right)+a\left(t-\hat{t}\right)+p\cdot\left(x-\hat{x}\right)+\frac{1}{2}\left(x-\hat{x}\right)^{T}\cdot X\cdot\left(x-\hat{x}\right)+o\left(\left\vert \hat{x}-x\right\vert ^{2}+\left\vert \hat{t}-t\right\vert \right)\label{EqTaylor}
\end{equation}
as $\left\vert \hat{x}-x\right\vert ^{2}+\left\vert \hat{t}-t\right\vert \rightarrow0$
with $a=\partial_{t}\varphi\left(\hat{t},\hat{x}\right),$ $p=D\varphi\left(\hat{t},\hat{x}\right),$
$X=D^{2}\varphi\left(\hat{t},\hat{x}\right)$. Moreover, if (\ref{EqTaylor})
holds for some $\left(a,p,X\right)$ and $u$ is differentiable, then
$a=\partial_{t}u\left(\hat{t},\hat{x}\right),$ $p=Du\left(\hat{t},\hat{x}\right),$
$X\leq D^{2}u\left(\hat{t},\hat{x}\right)$, hence by degenerate ellipticity
\[
\partial_{t}\varphi+F\left(\hat{t},\hat{x},u,p,X\right)\leq0\text{.}
\]
Pushing this idea further leads to a definition of viscosity solutions
based on a generalized notion of ``$\left(\partial_{t}u,Du,D^{2}u\right)$''
for non-differentiable $u$, the so-called parabolic semijets, and
it is a simple exercise to show that both definitions are equivalent.
The resulting theory (existence, uniqueness, stability, ...) is without
doubt one of the most important recent developments in the field of
partial differential equations. As a typical result%
\footnote{$\mathrm{BUC}\left(\dots\right)$ denotes the space of bounded, uniformly
continuous functions.%
}, the initial value problem $\left(\partial_{t}+F\right)u=0,\, u\left(0,\cdot\right)=u_{0}\in\mathrm{BUC}\left(\mathbb{R}^{n}\right)$
has a unique solution in $\mathrm{BUC}\left([0,T]\times\mathbb{R}^{n}\right)$
provided $F=F(t,x,u,Du,D^{2}u)$ is continuous, degenerate elliptic,
proper (i.e. increasing in the $u$ variable) and satisfies a (well-known)
technical condition%
\footnote{(3.14) of the User's Guide \cite{MR1118699UserGuide}.%
}. In fact, uniqueness follows from a stronger property known as \textit{comparison:}
assume $u$ (resp.\ $v$) is a supersolution (resp.\ subsolution)
and $u_{0}\geq v_{0}$; then $u\geq v$ on $[0,T]\times\mathbb{R}^{n}$.
A key feature of viscosity theory is what workers in the field simply
call \textit{stability properties}. For instance, it is relatively
straightforward to study $\left(\partial_{t}+F\right)u=0$ via a sequence
of approximate problems, say $\left(\partial_{t}+F^{n}\right)u^{n}=0$,
provided $F^{n}\rightarrow F$ locally uniformly and some apriori
information on the $u^{n}$ (e.g.\ locally uniform convergence, or
locally uniform boundedness%
\footnote{What we have in mind here is the \textit{Barles--Perthame method of
semi-relaxed limits} \cite{MR2179357FS}.%
}. Note the stark contrast to the classical theory where one has to
control the actual derivatives of $u^{n}$.

The idea of stability is also central to \textit{rough path theory}.
Given a collection $\left(V_{1},\dots,V_{d}\right)$ of (sufficiently
nice) vector fields on $\mathbb{R}^{n}$ and $z\in C^{1}\left(\left[0,T\right],\mathbb{R}^{d}\right)$
one considers the (unique) solution $y$ to the ordinary differential
equation
\begin{equation}
\dot{y}\left(t\right)=\sum_{i=1}^{d}V_{i}\left(y\right)\dot{z}^{i}\left(t\right),\,\,\, y\left(0\right)=y_{0}\in\mathbb{R}^{n}\text{.}\label{ODEintro}
\end{equation}
The question is, if the output signal $y$ depends in a stable way
on the driving signal $z$ (one handles time-dependent vector fields
$V=V\left(t,y\right)$ by considering the $\left(d+1\right)$-dimensional
driving signal $t\mapsto\left(t,z_{t}\right)$). The answer, of course,
depends strongly on how to measure distance between input signals.
If one uses the supremums norm, so that the distance between driving
signals $z,\tilde{z}$ is given by $\left\vert z-\tilde{z}\right\vert _{\infty;\left[0,T\right]}=\sup_{r\in\left[0,T\right]}\left|z_{r}-\tilde{z}_{r}\right|$,
then the solution will in general \textit{not }depend continuously
on the input.

\begin{example} \label{ODEnotContInInfity}Take $n=1,d=2,\, V=\left(V_{1},V_{2}\right)=\left(\sin\left(\cdot\right),\cos\left(\cdot\right)\right)$
and $y_{0}=0$. Obviously,
\[
z^{n}\left(t\right)=\left(\frac{1}{n}\cos\left(2\pi n^{2}t\right),\frac{1}{n}\sin\left(2\pi n^{2}t\right)\right)
\]
converges to $0$ in $\infty$-norm whereas the solutions to $\dot{y}^{n}=V\left(y^{n}\right)\dot{z}^{n},y_{0}^{n}=0,$
do not converge to zero (the solution to the limiting equation $\dot{y}=0$).
\end{example}

If $\left\vert z-\tilde{z}\right\vert _{\infty;\left[0,T\right]}$
is replaced by the (much) stronger distance 
\[
\left\vert z-\tilde{z}\right\vert _{1\text{-var};\left[0,T\right]}=\sup_{\left(t_{i}\right)\subset\left[0,T\right]}\sum\left\vert z_{t_{i},t_{i+1}}-\tilde{z}_{t_{i},t_{i+1}}\right\vert ,
\]
(using the notation $z_{s,t}:=z_{t}-z_{s}$) it is elementary to see
that now the solution map is continuous (in fact, locally Lipschitz);
however, this continuity does not lend itself to push the meaning
of (\ref{ODEintro}): the closure of $C^{1}$ (or smooth) paths in
variation is precisely $W^{1,1}$, the set of absolutely continuous
paths (and thus still far from a typical Brownian path). Lyons' theory
of rough paths exhibits an entire cascade of ($p$-variation or $1/p$-Hölder
type rough path) metrics, for each $p\in\lbrack1,\infty)$, on path-space
under which such ODE solutions are continuous (and even locally Lipschitz)
functions of their driving signal. For instance, the ``rough path''
$p$-variation distance between two smooth $\mathbb{R}^{d}$-valued
paths $z,\tilde{z}$ is given by
\[
\max_{j=1,\dots,\left[p\right]}\left(\sup_{\left(t_{i}\right)\subset\left[0,T\right]}\sum\left\vert z_{t_{i},t_{i+1}}^{\left(j\right)}-\tilde{z}_{t_{i},t_{i+1}}^{\left(j\right)}\right\vert ^{p}\right)^{1/p}
\]
where $z_{s,t}^{\left(j\right)}=\int dz_{r_{1}}\otimes\dots\otimes dz_{r_{j}}$
with integration over the $j$-dimensional simplex $\left\{ s<r_{1}<\dots<r_{j}<t\right\} $.
This allows to extend the very meaning of (\ref{ODEintro}), in a
unique and continuous fashion, to driving signals which live in the
abstract completion of smooth $\mathbb{R}^{d}$-valued paths (with
respect to rough path $p$-variation or a similarly defined $1/p$-Hölder
metric). The space of so-called $p$-rough paths%
\footnote{In the strict terminology of rough path theory: geometric $p$-rough
paths.%
} is precisely this abstract completion. In fact, this space can be
realized as genuine path space, where $G^{\left[p\right]}\left(\mathbb{R}^{d}\right)$
is the free step-$\left[p\right]$ nilpotent group over $\mathbb{R}^{d}$,
equipped with Carnot--Caratheodory metric; realized as a subset of
$1+\mathfrak{t}^{\left[p\right]}\left(\mathbb{R}^{d}\right)$ where
\[
\mathfrak{t}^{\left[p\right]}\left(\mathbb{R}^{d}\right)=\mathbb{R}^{d}\oplus\left(\mathbb{R}^{d}\right)^{\otimes2}\oplus\dots\oplus\left(\mathbb{R}^{d}\right)^{\otimes\left[p\right]}
\]
is the natural space for (up to $\left[p\right]$) iterated integrals
of a smooth $\mathbb{R}^{d}$-valued path. For instance, almost every
realization of $d$-dimensional Brownian motion $B$ \textit{enhanced
with its iterated stochastic integrals in the sense of\ Stratonovich,
}i.e. the matrix-valued process given by
\begin{equation}
B^{\left(2\right)}:=\left(\int_{0}^{\cdot}B^{i}\circ dB^{j}\right)_{i,j\in\left\{ 1,\dots,d\right\} }\label{StratoII_intro}
\end{equation}
yields a path $\mathbf{B}\left(\omega\right)$ in $G^{2}\left(\mathbb{R}^{d}\right)$
with finite $1/p$-Hölder (and hence finite $p$-variation) regularity,
for any $p>2$. ($\mathbf{B}$ is known as \textit{Brownian rough
path.}) We remark that $B^{\left(2\right)}=\frac{1}{2}B\otimes B+A$
where the anti-symmetric part of the matrix, $A:=\mathrm{Anti}\left(B^{\left(2\right)}\right)$,
is known as \textit{Lévy's stochastic area;} in other words $\mathbf{B}\left(\omega\right)$
is determined by $\left(B,A\right)$, i.e. Brownian motion \textit{enhanced
with Lévy's area. }A similar construction work when $B$ is replaced
by a generic multi-dimensional continuous semimartingales; see \cite[Chapter 14]{friz-victoir-book}
and the references therein.

\section{Transformations}

\label{sec:trans}

\subsection{Inner and outer transforms}

Throughout, $F=F\left(t,x,r,p,X\right)$ is a continuous scalar-valued
function on $\left[0,T\right]\times\mathbb{R}^{n}\times\mathbb{R}\times\mathbb{R}^{n}\times\mathbb{S}\left(n\right)$,
$\mathbb{S}\left(n\right)$ denotes the space of symmetric $n\times n$-matrices,
and $F$ is assumed to be non-increasing in $X$ (degenerate elliptic)
and proper in the sense of \eqref{WeakProper}. Time derivatives of
functions are denoted by upper dots, spatial derivatives (with respect
to $x$) by $D,D^{2}$, etc. Further, we use \,$\left\langle .,.\right\rangle $
to denote tensor contraction%
\footnote{We also use $\cdot$ to denote contraction over only index or to denote
matrix multiplication.%
}, i.e.\ $\left\langle p,q\right\rangle _{j_{1},\ldots,j_{n}}\equiv\sum_{i_{1},\ldots,i_{m}}p_{_{i_{1},\ldots,i_{m}}}q_{j_{1},\ldots,j_{n}}^{_{i_{1},\ldots,i_{m}}},$
$p\in\left(\mathbb{R}^{l}\right)^{\otimes m},q\in\left(\mathbb{R}^{l}\right)^{\otimes n}\otimes\left(\left(\mathbb{R}^{l}\right)^{\prime}\right)^{\otimes m}$.

\begin{lemma}{[}Inner Transform{]} \label{Lemma_Gradient_Transform}Let
$z\in C^{1}\left(\left[0,T\right],\mathbb{R}^{d}\right),$ $\sigma=\left(\sigma_{1},\ldots,\sigma_{d}\right)\subset C_{b}^{2}\left(\left[0,T\right]\times\mathbb{R}^{n},\mathbb{R}^{n}\right)$
(the space of continuous and twice differentiable, bounded functions
with bounded derivatives) and $\psi=\psi\left(t,x\right)$ the ODE
flow of $dy=\sigma\left(y\right)dz,$ i.e.
\[
\dot{\psi}\left(t,x\right)=\sum_{i=1}^{d}\sigma_{i}\left(t,\psi\left(t,x\right)\right)\dot{z}_{t}^{i},\text{ }\dot{\psi}\left(0,x\right)=x\in\mathbb{R}^{n}\text{.}
\]
Then $u$ is a viscosity subsolution (always assumed $\mathrm{BUC}$)
of 
\begin{equation}
\partial_{t}u+F\left(t,x,r,Du,D^{2}u\right)-\sum_{i=1}^{d}\left(Du\cdot\sigma_{i}\left(t,x\right)\right)\dot{z}_{t}^{i}=0;\,\,\, u\left(0,.\right)=u_{0}\left(.\right)\label{eq for u}
\end{equation}
iff $w\left(t,x\right):=u\left(t,\psi\left(t,x\right)\right)$ is
a viscosity subsolution of 
\begin{equation}
\partial_{t}w+F^{\psi}\left(t,x,w,Dw,D^{2}w\right)=0;\,\,\, w\left(0,.\right)=u_{0}\left(.\right)\label{eq for v}
\end{equation}
where 
\[
F^{\psi}\left(t,x,r,p,X\right)=F\left(t,\psi_{t}\left(x\right),r,\left\langle p,D\psi_{t}^{-1}|_{\psi_{t}\left(x\right)}\right\rangle ,\left\langle X,D\psi_{t}^{-1}|_{\psi_{t}\left(x\right)}\otimes D\psi_{t}^{-1}|_{\psi_{t}\left(x\right)}\right\rangle +\left\langle p,D^{2}\psi_{t}^{-1}|_{\psi_{t}\left(x\right)}\right\rangle \right)
\]
and
\[
D\psi_{t}^{-1}|_{x}=\left(\frac{\partial\left(\psi_{t}^{-1}\left(t,x\right)\right)^{k}}{\partial x^{i}}\right)_{i=1,\dots,n}^{k=1,\dots,n}\text{ and }D^{2}\psi_{t}^{-1}|_{x}=\left(\frac{\partial\left(\psi^{-1}\left(t,x\right)\right)^{k}}{\partial x^{i}x^{j}}\right)_{i,j=1,\dots,n}^{k=1,\dots,n}.
\]
The same statement holds if one replaces the word subsolution by supersolution
throughout. \end{lemma}

\begin{remark}\label{rem.timereg}The regularity assumptions on $\sigma$
with respect to $t$ can be obviously relaxed here. Treating time
and space variable similarly will be convenient in the rough path
framework where sharp results on time-dependent vector fields are
hard to find in the literature (but see \cite{CFG}). \end{remark}

If we specialize from general $F$ to a semilinear $L:\left[0,T\right]\times\mathbb{R}^{n}\times\mathbb{R}\times\mathbb{R}^{n}\times\mathbb{S}\left(n\right)\rightarrow\mathbb{R}$
we get transformation 1 as a corollary.

\begin{corollary}{[}Transformation 1{]} Let $\psi=\psi\left(t,x\right)$
be the ODE flow of $dy=\sigma\left(t,y\right)dz$, as above. Define
$L=L\left(t,x,r,p,X\right)$ by
\[
L=-\text{Tr}\left[A\left(t,x\right)\cdot X\right]+b\left(t,x\right)\cdot p+c\left(t,x,r\right);
\]
define also the transform
\[
L^{\psi}=-\text{Tr}\left[A^{\psi}\left(t,x\right)\cdot X\right]+b^{\psi}\left(t,x\right)\cdot p+c^{\psi}\left(t,x,r\right)
\]
where
\begin{eqnarray*}
A^{\psi}\left(t,x\right) & = & \left\langle A\left(t,\psi_{t}\left(x\right)\right),D\psi_{t}^{-1}|_{\psi_{t}\left(x\right)}\otimes D\psi_{t}^{-1}|_{\psi_{t}\left(x\right)}\right\rangle ,\\
b^{\psi}\left(t,x\right)\cdot p & = & b\left(t,\psi_{t}\left(x\right)\right)\cdot\left\langle p,D\psi_{t}^{-1}|_{\psi_{t}\left(x\right)}\right\rangle -\text{Tr}\left(A\left(t,\psi_{t}\right)\cdot\left\langle p,D^{2}\psi_{t}^{-1}|_{\psi_{t}\left(x\right)}\right\rangle \right),\\
c^{\psi}\left(t,x,r\right) & = & c\left(t,\psi_{t}\left(x\right),r\right).
\end{eqnarray*}
Then $u$ is a solution (always assumed $BUC$) of 
\[
\partial_{t}u+L\left(t,x,u,Du,D^{2}u\right)=\sum_{i=1}^{d}\left(Du\cdot\sigma_{i}\left(t,x\right)\right)\dot{z}_{t}^{i};\,\,\, u\left(0,.\right)=u_{0}\left(.\right)
\]
if and only if $u^{1}\left(t,x\right):=u\left(t,\psi\left(t,x\right)\right)$
is a solution of 
\begin{equation}
\partial_{t}+L^{\psi}=0;\,\,\, u^{1}\left(0,.\right)=u_{0}\left(.\right)
\end{equation}
\end{corollary}

\begin{proof}[Proof of Lemma \ref{Lemma_Gradient_Transform}]  Set
$y=\psi_{t}\left(x\right)$. When $u$ is a classical sub-solution,
it suffices to use the chain rule and definition of $F^{\psi}$ to
see that
\begin{eqnarray*}
\dot{w}\left(t,x\right) & = & \dot{u}\left(t,y\right)+Du\left(t,y\right)\cdot\dot{\psi}_{t}\left(x\right)=\dot{u}\left(t,y\right)+Du\left(t,y\right)\cdot\sigma\left(y\right)\dot{z}_{t}\\
 & \leq & F\left(t,y,u\left(t,y\right),Du\left(t,y\right),D^{2}u\left(t,y\right)\right)=F^{\psi}\left(t,x,w\left(t,x\right),Dw\left(t,x\right),D^{2}w\left(t,x\right)\right).
\end{eqnarray*}
The case when $u$ is a viscosity sub-solution of (\ref{eq for u})
is not much harder: suppose that $\left(\bar{t},\bar{x}\right)$ is
a maximum of $w-\xi$, where $\xi\in C^{2}\left(\left[0,T\right]\times\mathbb{R}^{n}\right)$
and define $\varphi\in C^{2}\left(\left(0,T\right)\times\mathbb{R}^{n}\right)$
by $\varphi\left(t,y\right)=\xi\left(t,\psi_{t}^{-1}\left(y\right)\right)$.
Set $\bar{y}=\psi_{\bar{t}}\left(\bar{x}\right)$ so that
\[
F\left(\bar{t},\bar{y},u\left(\bar{t},\bar{y}\right),D\varphi\left(\bar{t},\bar{y}\right),D^{2}\varphi\left(\bar{t},\bar{y}\right)\right)=F^{\psi}\left(\bar{t},\bar{x},w\left(\bar{t},\bar{x}\right),D\xi\left(\bar{t},\bar{x}\right),D^{2}\xi\left(\bar{t},\bar{x}\right)\right).
\]
Obviously, $\left(\bar{t},\bar{y}\right)$ is a maximum of $u-\varphi$,
and since $u$ is a viscosity sub-solution of (\ref{eq for u}) we
have
\[
\dot{\varphi}\left(\bar{t},\bar{y}\right)+D\varphi\left(\bar{t},\bar{y}\right)\sigma\left(\bar{t},\bar{y}\right)\dot{z}\left(\bar{t}\right)\leq F\left(\bar{t},\bar{y},u\left(\bar{t},\bar{y}\right),D\varphi\left(\bar{t},\bar{y}\right),D^{2}\varphi\left(\bar{t},\bar{y}\right)\right).
\]
On the other hand, $\xi\left(t,x\right)=\varphi\left(t,\psi_{t}\left(x\right)\right)$
implies $\dot{\xi}\left(\bar{t},\bar{x}\right)=\dot{\varphi}\left(\bar{t},\bar{y}\right)+D\varphi\left(\bar{t},\bar{y}\right)\sigma\left(\bar{t},\bar{y}\right)\dot{z}\left(\bar{t}\right)$
and putting things together we see that
\[
\dot{\xi}\left(\bar{t},\bar{x}\right)\leq F^{\psi}\left(\bar{t},\bar{x},w\left(\bar{t},\bar{x}\right),D\xi\left(\bar{t},\bar{x}\right),D^{2}\xi\left(\bar{t},\bar{x}\right)\right)
\]
which says precisely that $w$ is a viscosity sub-solution of (\ref{eq for v}).
Replacing maximum by minimum and $\leq$ by $\geq$ in the preceding
argument, we see that if $u$ is a super-solution of (\ref{eq for u}),
then $w$ is a super-solution of (\ref{eq for v}).\\
Conversely, the same arguments show that if $v$ is a viscosity sub-
(resp.~super-) solution for (\ref{eq for v}), then $u\left(t,y\right)=w\left(t,\psi^{-1}\left(y\right)\right)$
is a sub- (resp.~super-) solution for (\ref{eq for u}). \end{proof}

We prepare the next lemma by agreeing that for a sufficiently smooth
function $\phi=\phi\left(t,r,x\right):\left[0,T\right]\times\mathbb{R}\times\mathbb{R}^{n}\rightarrow\mathbb{R}$
we shall write 
\begin{eqnarray*}
\dot{\phi} & = & \frac{\partial\phi\left(t,r,x\right)}{\partial t},\phi^{\prime}=\frac{\partial\phi\left(t,r,x\right)}{\partial r},\\
D\phi & = & \left(\frac{\partial\phi\left(t,r,x\right)}{\partial x^{i}}\right)_{i=1,\dots,n}\text{ and }D^{2}\phi=\left(\frac{\partial^{2}\phi\left(t,r,x\right)}{\partial x^{i}\partial x^{j}}\right)_{i,j=1,\dots,n}.
\end{eqnarray*}

\begin{lemma}{[}Outer transform{]} \label{Lemma_u_Transform}Let
$\phi=\phi\left(t,r,x\right)\in C^{1,2,2}$ and assume that $\forall\left(t,x\right)$,
$r\mapsto\phi\left(t,r,x\right)$ is an increasing diffeomorphism
on the real line. Then $u$ is a subsolution of $\partial_{t}u+F\left(t,x,u,Du,D^{2}u\right)=0,$
$u\left(0,.\right)=u_{0}\left(.\right)$ if and only if
\[
v\left(t,x\right)=\phi^{-1}\left(t,u\left(t,x\right),x\right)
\]
is a subsolution of $\partial_{t}v+^{\phi}F\left(t,x,v,Dv,D^{2}v\right)=0,$
$v\left(0,.\right)=\phi^{-1}\left(0,u_{0}\left(x\right),x\right)$
with 
\begin{eqnarray}
^{\phi}F\left(t,x,r,p,X\right) & = & \frac{\dot{\phi}}{\phi^{\prime}}+\frac{1}{\phi^{\prime}}F\left(t,x,\phi,D\phi+\phi^{\prime}p,\right.\label{EqFPhi}\\
 &  & \left.\phi^{\prime\prime}p\otimes p+D\phi^{\prime}\otimes p+p\otimes D\phi^{\prime}+D^{2}\phi+\phi^{\prime}X\right)\notag
\end{eqnarray}
where $\phi$ and all derivatives are evaluated at $\left(t,r,x\right)$.
The same statement holds if one replaces the word subsolution by supersolution
throughout. \end{lemma}

\begin{proof} $\left(\Longrightarrow\right)$ We show the first implication,
i.e. assume $u$ is a subsolution of $\partial_{t}u+F=0$ and set
$v\left(t,x\right)=\phi^{-1}\left(t,u\left(t,x\right),x\right)$.
By definition, $\left(a,p,X\right)\in\mathcal{P}^{2;+}v\left(s,z\right)$
(the parabolic superjet, cf.~\cite[section 8]{MR1118699UserGuide})
iff
\[
v\left(t,x\right)\leq v\left(s,z\right)+a\left(t-s\right)+p\cdot\left(x-z\right)+\frac{1}{2}\left(x-z\right)^{T}\cdot X\cdot\left(x-z\right)+o\left(\left\vert t-s\right\vert +\left\vert x-z\right\vert ^{2}\right)
\]
as $\left(t,x\right)\rightarrow\left(s,z\right)\text{.}$Since $\phi\left(t,.,x\right)$
is increasing, 
\[
\phi\left(t,v\left(t,x\right),x\right)\leq\phi\left(t,*,x\right)
\]
with 
\[
*=v\left(s,z\right)+a\left(t-s\right)+p\cdot\left(x-z\right)+\frac{1}{2}\left(x-z\right)^{T}\cdot X\cdot\left(x-z\right)+o\left(\left\vert t-s\right\vert +\left\vert x-z\right\vert ^{2}\right)
\]
and using a Taylor expansion on $\phi$ in all three arguments we
see that the right hand side equals 
\begin{eqnarray*}
 &  & \phi\left(s,v\left(s,z\right),z\right)+\dot{\phi}_{s,v\left(s,z\right),z}\left(t-s\right)+\phi_{s,v\left(s,z\right),z}^{\prime}a\left(t-s\right)+\phi_{s,v\left(s,z\right),z}^{\prime}p\cdot\left(x-z\right)\\
 &  & +\frac{1}{2}\phi_{s,v\left(s,z\right),z}^{\prime}\left(x-z\right)^{T}\cdot X\cdot\left(x-z\right)+D\phi_{s,v\left(s,z\right),z}\cdot\left(x-z\right)+\frac{1}{2}\left(x-z\right)^{T}\cdot D^{2}\phi_{s,v\left(s,z\right),z}\cdot\left(x-z\right)\\
 &  & +\left(x-z\right)^{T}\cdot\left(D\left(\phi^{\prime}\right)\right)_{s,v\left(s,z\right),z}\otimes p\cdot\left(x-z\right)\\
 &  & +\left(x-z\right)^{T}\cdot p\otimes\left(D\phi\right)_{s,v\left(s,z\right),z}^{\prime}\cdot\left(x-z\right)\\
 &  & +\left(x-z\right)^{T}\cdot\phi_{s,v\left(s,z\right),z}^{\prime\prime}p\otimes p\cdot\left(x-z\right)+o\left(\left\vert t-s\right\vert +\left\vert x-z\right\vert ^{2}\right)\text{ as }\left(s,z\right)\rightarrow\left(t,x\right)
\end{eqnarray*}
Hence, 
\begin{eqnarray*}
 &  & \left(\dot{\phi}_{s,v\left(s,z\right),z}+\phi_{s,v\left(s,z\right),z}^{\prime}a,D\phi_{s,v\left(s,z\right),z}+\phi_{s,v\left(s,z\right),z}^{\prime}p,\right.\\
 &  & \left.\phi_{s,v\left(s,z\right),z}^{\prime\prime}p\otimes p+D\left(\phi^{\prime}\right)_{s,v\left(s,z\right),z}\otimes p+p\otimes\left(D\phi\right)_{s,v\left(s,z\right),z}^{\prime}+D^{2}\phi_{s,v\left(s,z\right),z}+\phi_{s,v\left(s,z\right),z}^{\prime}X\right)
\end{eqnarray*}
belongs to $\mathcal{P}^{2;+}u\left(s,z\right)$ and since $u$ is
a subsolution this immediately shows 
\begin{eqnarray*}
\dot{\phi}_{s,v\left(s,z\right),z}+\phi_{\left(s,v\left(s,z\right),z\right)}^{\prime}a+F\left(s,z,\phi_{\left(s,v\left(s,z\right),z\right)},D\phi_{s,v\left(s,z\right),z}+\phi_{s,v\left(s,z\right),z}^{\prime}p,\right.\\
\left.\phi_{s,v\left(s,z\right),z}^{\prime\prime}p\otimes p+D\left(\phi^{\prime}\right)_{s,v\left(s,z\right),z}\otimes p+p\otimes\left(D\phi\right)_{s,v\left(s,z\right),z}^{\prime}+D^{2}\phi_{s,v\left(s,z\right),z}+\phi_{s,v\left(s,z\right),z}^{\prime}X\right) & \leq & 0\text{.}
\end{eqnarray*}
Dividing by $\phi^{\prime}\ $($>0$) shows that $v$ is a subsolution
of $\partial_{t}v+F^{\phi}=0$.

$\left(\Longleftarrow\right)$ Assume $v$ is a subsolution of $\partial_{t}v+^{\phi}F=0$,
$^{\phi}F$ defined as in $\left(\ref{EqFPhi}\right)$ for some $F.$
Set $u\left(t,x\right):=\phi\left(t,v\left(t,x\right),x\right)$.
By above argument we know that $v$ is a subsolution of $^{\phi^{-1}}\left(^{\phi}F\right)\left(t,x,r,p,X\right).$
For brevity write $\psi\left(t,.,x\right)=\phi^{-1}\left(t,.,x\right).$
Then
\begin{eqnarray*}
 &  & \,^{\phi^{-1}}\left(^{\phi}F\right)\left(t,x,r,p,X\right)\\
 & = & \frac{\psi_{t,r,x}}{\psi_{t,r,x}^{\prime}}+\frac{1}{\psi_{t,r,x}^{\prime}}\,^{\phi}F\left(t,x,\psi_{\left(t,r,x\right)},D\psi_{t,r,x}+\psi_{t,r,x}^{\prime}p,\right.\\
 &  & \left.\psi_{t,r,x}^{\prime\prime}p\otimes p+D\left(\psi^{\prime}\right)_{t,r,x}\otimes p+p\otimes\left(D\psi\right)_{t,r,x}^{\prime}+D^{2}\psi_{t,r,x}+\psi_{t,r,x}^{\prime}X\right)\\
 & = & \frac{\psi_{t,r,x}}{\psi_{t,r,x}^{\prime}}+\frac{1}{\psi_{t,r,x}^{\prime}}\left[\frac{\dot{\phi}_{t,\psi_{t,r,x},x}}{\phi_{t,\psi_{t,r,x},x}^{\prime}}+\right.\\
 &  & \frac{1}{\phi_{t,\psi_{t,r,x},x}^{\prime}}F\left(t,x,\phi\left(t,\psi_{t,r,x},x\right),D\phi_{t,\psi_{t,r,x},x}+\phi_{t,\psi_{t,r,x},x}^{\prime}\left\{ D\psi_{t,r,x}+\psi_{t,r,x}^{\prime}p\right\} \right.,\\
 &  & \phi_{t,\psi_{t,r,x},x}^{\prime\prime}p\otimes p+D\left(\phi^{\prime}\right)_{t,\psi_{t,r,x},x}\otimes p+p\otimes\left(D\phi\right)_{t,\psi_{t,r,x},x}^{\prime}+D^{2}\phi_{t,\psi_{t,r,x},x}\\
 &  & \left.\left.+\phi_{t,\psi_{t,r,x},x}^{\prime}\left\{ \psi_{t,r,x}^{\prime\prime}p\otimes p+D\left(\psi^{\prime}\right)_{t,r,x}\otimes p+p\otimes\left(D\psi\right)_{t,r,x}^{\prime}+D^{2}\psi_{t,r,x}+\psi_{t,r,x}^{\prime}X\right\} \right)\right].
\end{eqnarray*}
Using several times equalities of the type $\left(f\circ f^{-1}\right)^{\prime}=f_{f^{-1}}^{\prime}\left(f^{-1}\right)^{\prime}=id$
cancels the terms involving $\phi,\psi$ and their derivatives and
we are left with $F$, i.e.
\[
\,^{\phi^{-1}}\left(^{\phi}F\right)=F\text{.}
\]
This finishes the proof. \end{proof}

\begin{corollary}{[}Transformation 2{]} Assume $\nu=\left(\nu_{1},\dots,\nu_{d}\right)\subset C_{b}^{0,2}\left(\left[0,T\right]\times\mathbb{R}^{n}\right)$
(i.e.~continuous, bounded and twice differentiable in the second
variable with bounded derivatives). Assume $\phi=\phi\left(t,x,r\right)$
is determined by the ODE
\[
\dot{\phi}=\phi\,\sum_{j=1}^{d}\nu_{j}\left(t,x\right)\dot{z}_{t}^{j}\equiv\phi\,\nu\left(t,x\right)\cdot\dot{z}_{t},\,\,\,\phi\left(0,x,r\right)=r.
\]
Define $L=L\left(t,x,r,p,X\right)$ by
\[
L=-\text{Tr}\left[A\left(t,x\right)\cdot X\right]+b\left(t,x\right)\cdot p+c\left(t,x,r\right);
\]
define also
\begin{eqnarray}
^{\phi}L\left(t,x,r,X\right) & = & -\text{Tr}\left[A\left(t,x\right)\cdot X\right]+\,^{\phi}b\left(t,x\right)\cdot p+\,^{\phi}c\left(t,x,r\right)\label{phiL}
\end{eqnarray}
where 
\begin{align*}
^{\phi}b\left(t,x\right)\cdot p & \equiv b\left(t,x\right)\cdot p-\frac{2}{\phi^{\prime}}\text{Tr}\left[A\left(t,x\right)\cdot D\phi^{\prime}\otimes p\right]\\
^{\phi}c\left(t,x,r\right) & \equiv-\frac{1}{\phi^{\prime}}\text{Tr}\left[A\left(t,x\right)\cdot(D^{2}\phi)\right]+\frac{1}{\phi^{\prime}}b\left(t,x\right)\cdot\left(D\phi\right)+\frac{1}{\phi^{\prime}}c\left(t,x,\phi\right)
\end{align*}
with $\phi$ and all its derivatives evaluated at $\left(t,r,x\right)$.
Then
\[
\partial_{t}w+L\left(t,x,w,Dw,D^{2}w\right)-w\,\nu\left(t,x\right)\cdot\dot{z}\left(t\right)=0
\]
if and only if $v\left(t,x\right)=\phi^{-1}\left(t,w\left(t,x\right),x\right)$
satisfies
\[
\partial_{t}v+^{\phi}L\left(t,x,v,Dv,D^{2}v\right)=0.
\]
\end{corollary}

\begin{proof} Obviously, 
\[
\phi\left(t,x,r\right)=r\exp\left(\int_{0}^{t}\sum_{j=1}^{d}\nu_{j}\left(s,x\right)\dot{z}_{s}^{j}\right).
\]
This implies that $\phi^{\prime}=\phi/r$ and $D\phi^{\prime}$ do
not depend on $r$ so that indeed $^{\phi}b\left(t,x\right)$ defined
above has no $r$ dependence. Also note that $\phi^{\prime\prime}=0$
and $\dot{\phi}/\phi=d\cdot\dot{z}\equiv\sum_{j=1}^{d}d_{j}\left(t,x\right)\dot{z}_{t}^{j}$.
It follows, for general $F$, that
\begin{eqnarray*}
^{\phi}F\left(t,x,r,p,X\right) & = & r\, d\cdot\dot{z}+\frac{1}{\phi^{\prime}}F\left(t,x,\phi,D\phi+\phi^{\prime}p,\right.\\
 &  & \left.D\phi^{\prime}\otimes p+p\otimes D\phi^{\prime}+D^{2}\phi+\phi^{\prime}X\right)
\end{eqnarray*}
and specializing to $F=L-w\nu\cdot\dot{z}$, of the assumed (semi-)
linear form, we see that
\begin{eqnarray*}
^{\phi}L & = & -\frac{1}{\phi^{\prime}}\text{Tr}\left[A\left(t,x\right)\cdot(D\phi^{\prime}\otimes p+p\otimes D\phi^{\prime}+D^{2}\phi+\phi^{\prime}X)\right]\\
 &  & +\frac{1}{\phi^{\prime}}b\left(t,x\right)\cdot\left(D\phi+\phi^{\prime}p\right)+\frac{1}{\phi^{\prime}}c\left(t,x,\phi\right)
\end{eqnarray*}
where $\phi$ and all derivatives are evaluated at $\left(t,r,x\right).$
Observe that $^{\phi}L$ is again linear in $X$ and $p$. It now
suffices to collect the corresponding terms to obtain (\ref{phiL}).
\end{proof}

We shall need another (outer)transform to remove additive noise.

\begin{lemma}{[}Transformation 3{]} Let $g\in C\left(\left[0,T\right]\times\mathbb{R}^{n},\mathbb{R}^{d}\right)$
and set $\varphi\left(t,x\right)=\int_{0}^{t}g\left(s,x\right)dz_{s}=\sum_{i=1}^{d}\int_{0}^{t}g_{i}\left(s,x\right)dz_{s}^{i}$.
Define
\begin{eqnarray*}
L\left(t,x,r,p,X\right) & = & -\text{Tr}\left[A\left(t,x\right)\cdot X\right]+b\left(t,x\right)\cdot p+c\left(t,x,r\right);\\
L_{\varphi}\left(t,x,r,p,X\right) & = & -\text{Tr}\left[A\left(t,x\right)\cdot X\right]+b\left(t,x\right)\cdot p+c_{\varphi}\left(t,x,r\right)\\
\text{with }c_{\varphi}\left(t,x,r\right) & = & \text{Tr}\left[A\left(t,x\right)\cdot D^{2}\varphi\left(t,x\right)\right]-b\left(t,x\right)\cdot D\varphi\left(t,x\right)+c\left(t,x,r-\varphi\left(t,x\right)\right)
\end{eqnarray*}
Then $v$ solves
\[
\partial_{t}v+L\left(t,x,v,Dv,D^{2}v\right)-g\left(t,x\right)\cdot\dot{z}\left(t\right)=0
\]
if and only if $\tilde{v}\left(t,x\right)=v\left(t,x\right)+\varphi\left(t,x\right)$
solves
\[
\partial_{t}\tilde{v}+L_{\varphi}\left(t,x,\tilde{v},D\tilde{v},D^{2}\tilde{v}\right)=0.
\]
\end{lemma}

\begin{proof} Left to reader. \end{proof}

\subsection{The full transformation}

As before, let
\[
L\left(t,x,r,p,X\right):=-\text{Tr}\left[A\left(t,x\right)X\right]+b\left(t,x\right)\cdot p+c\left(t,x,r\right)
\]
where $A:\left[0,T\right]\times\mathbb{R}^{n}\rightarrow\mathbb{S}^{n}$,
$b:\left[0,T\right]\times\mathbb{R}^{n}\rightarrow\mathbb{R}^{n}$,
$f:\left[0,T\right]\times\mathbb{R}^{n}\times\mathbb{R}\rightarrow\mathbb{R}$.
Let us also define the following (linear, first order) differential
operators,
\begin{eqnarray}
M_{k}\left(t,x,u,Du\right) & = & \sigma_{k}\left(t,x\right)\cdot Du\text{ for }k=1,\dots,d_{1}\label{DefMk}\\
M_{d_{1}+k}\left(t,x,u,Du\right) & = & u\,\nu_{k}\left(t,x\right)\text{ for }k=1,\dots,d_{2}\notag\\
M_{d_{1}+d_{2}+k}\left(t,x,u,Du\right) & = & g_{k}\left(t,x\right)\text{ for }k=1,\dots,d_{3}.\notag
\end{eqnarray}
The combination of transformations 1,2 and 3 leads to the following

\begin{proposition} \label{Prop_Transform_to_PDE}Let $z^{1}\in C^{1}\left(\left[0,T\right],\mathbb{R}^{d_{1}}\right),$
$\sigma=\left(\sigma_{1},\ldots,\sigma_{d_{1}}\right)\subset C_{b}^{2}\left(\left[0,T\right]\times\mathbb{\mathbb{R}}^{n},\mathbb{\mathbb{R}}^{n}\right)$
and denote the ODE flow of $dy=\sigma\left(t,y\right)dz$ with $\psi,$
i.e. $\psi:\left[0,T\right]\times\mathbb{\mathbb{R}}^{n}\mathbb{\rightarrow R}^{n}$
satisfies
\begin{equation}
\dot{\psi}\left(t,x\right)=\sigma\left(t,\psi\left(t,x\right)\right)\dot{z}_{t}^{1}\text{, }\mathbb{\psi}\left(0,x\right)=x\in\mathbb{R}^{n}.\label{Transform1_psi}
\end{equation}
Further, let $z^{2}\in C^{1}\left(\left[0,T\right],\mathbb{R}^{d_{2}}\right)$
and let $\nu=\left(\nu_{1},\dots,\nu_{d_{2}}\right)$ be a collection
of $C_{b}^{0,2}\left(\left[0,T\right]\times\mathbb{R}^{n},\mathbb{\mathbb{R}}\right)$
functions and define $\phi=\phi\left(t,r,x\right)$ as solution to
the linear ODE 
\begin{equation}
\dot{\phi}=\phi\,\underset{\equiv d^{\psi}\left(t,x\right)}{\underbrace{\nu\left(t,\psi_{t}\left(x\right)\right)}}\dot{z}_{t}^{2},\,\,\,\phi\left(0,r,x\right)=r\in\mathbb{R}.\label{Transform2_phi}
\end{equation}
Further, let $z^{3}\in C^{1}\left(\left[0,T\right],\mathbb{R}^{d_{3}}\right)$
and define for given $g=\left(g_{1},\ldots,g_{d_{3}}\right)\in C\left(\left[0,T\right]\times\mathbb{R}^{n},\mathbb{R}^{d_{3}}\right)$,
$\varphi\left(t,x\right)$ as the integral%
\footnote{Since $\phi$ is linear in $r$, there is no $r$ dependence in its
derivative $\phi^{\prime}$.%
}
\begin{equation}
\varphi\left(t,x\right)=\int_{0}^{t}\,^{\phi}g^{\psi}\left(s,x\right)d\dot{z}_{s}^{3},\label{Transform3_a}
\end{equation}
\[
\text{where \thinspace}\,^{\phi}g^{\psi}\left(t,x\right)=\frac{1}{\phi^{\prime}\left(t,x\right)}g\left(t,\psi_{t}\left(x\right)\right).
\]
At last, set $z_{t}:=\left(z_{t}^{1},z_{t}^{2},z_{t}^{3}\right)\in\mathbb{R}^{d_{1}}\oplus\mathbb{R}^{d_{2}}\oplus\mathbb{R}^{d_{3}}\cong\mathbb{R}^{d}$.
Then $u$ is a viscosity solution of 
\begin{eqnarray}
\partial_{t}u+L\left(t,x,u,Du,D^{2}u\right) & = & \Lambda\left(t,x,u,Du\right)\dot{z}_{t},\text{ }\label{Eq_Linear_RPDE}\\
u\left(0,x\right) & = & u_{0}\left(x\right),
\end{eqnarray}
iff $\tilde{u}\left(t,x\right)=\phi^{-1}\left(t,u\left(t,\psi\left(t,x\right)\right),x\right)+\varphi\left(t,x\right)$
is a viscosity solution of
\begin{eqnarray}
\partial_{t}\tilde{u}+\tilde{L}\left(t,x,\tilde{u},D\tilde{u},D^{2}\tilde{u}\right) & = & 0\label{Eq_Transformed_to_PDE}\\
\tilde{u}\left(0,x\right) & = & u_{0}\left(x\right)
\end{eqnarray}
where $\tilde{L}=\,\left[^{\phi}\left(L^{\psi}\right)\right]_{\varphi}$
is obtained via transformations 1,2 and 3 (in the given order). \end{proposition}

\begin{remark} Transformation 2 and 3 could have been performed in
one step, by considering
\[
\dot{\phi}=\phi\,\nu^{\psi}\left(t,x\right)\cdot\dot{z}_{t}^{2}+g^{\psi}\left(t,x\right)\cdot\dot{z}_{t}^{3},\,\,\,\phi\left(t,r,x\right)|_{t=0}=r.
\]
Indeed, the usual variation of constants formula gives immediately
\[
\phi\left(t,x\right)=r\exp\left(\int_{0}^{t}\nu^{\psi}\left(s,x\right)dz_{s}^{2}\right)+\int_{0}^{t}e^{\left(\int_{s}^{t}\nu^{\psi}\left(\cdot,x\right)dz^{2}\right)}g^{\psi}\left(s,x\right)\cdot dz_{s}^{3}
\]
and one easily sees that transformations 2 and 3 just split above
expression in two terms; with the benefit of keeping the algebra somewhat
simpler (after all, we want explicit understandings of the transformed
equations). \end{remark}

\begin{remark}Related to the last remark, generic noise of the form
$H\left(t,x,u\right)dz$ can be removed with this technique. The issue
is that the transformed equations quickly falls beyond available viscosity
theory (e.g.~standard comparison results do no longer apply) cf.\ \cite{MR1799099,RBSDE}.
\end{remark}

\begin{proof} We first remove the terms driven by $z^{1}$: to this
end we apply transformation 1 with $L\left(t,x,r,p,X\right)$ replaced
by $L-r\nu\cdot\dot{z}^{2}-g\cdot\dot{z}^{3}$. The transformed solution,
$u^{1}\left(t,x\right)=u\left(t,\psi_{t}\left(x\right)\right)$, satisfies
the equation
\[
\left(\partial_{t}+L^{\psi}\right)u^{1}-u^{1}\underset{=d^{\psi}\left(t,x\right)}{\underbrace{\nu\left(t,\psi_{t}\left(x\right)\right)}}\cdot\dot{z}_{t}^{2}-\underset{=c^{\psi}\left(t,x\right)}{\underbrace{g\left(t,\psi_{t}\left(x\right)\right)}}\cdot\dot{z}_{t}^{3}=0
\]
We then remove the terms driven by $z^{2}$ by applying transformation
2 with $L^{\psi}-c^{\psi}\cdot\dot{z}^{3}$. The transformed solution
$u^{12}\left(t,x\right)=\phi^{-1}\left(t,u^{1}\left(t,x\right),x\right)$
satisfies the equation with operator 
\[
\left(\partial_{t}+\,^{\phi}\left(L^{\psi}-g^{\psi}\cdot\dot{z}^{3}\right)\right)
\]
i.e. 
\[
\partial_{t}u^{12}+\,^{\phi}(L^{\psi})u^{12}-\underset{=\,^{\phi}c^{\psi}}{\underbrace{\frac{1}{\phi^{\prime}}g^{\psi}}\cdot\dot{z}^{3}}=0.
\]
It now remains to apply transformation 3 to remove the remaining terms
driven by $z^{3}$. The transformed solution is precisely $\tilde{u}$,
as given in the statement of this proposition, and satisfies the equation
\[
\left(\partial_{t}+\,\left[^{\phi}\left(L^{\psi}\right)\right]_{\varphi}\right)\tilde{u}=0\text{.}
\]
The proof is now finished. \end{proof}

\subsection{Rough transformation\label{SectionRT}}

We need to understand transformations 1,2,3 when $\left(z^{1},z^{2},z^{3}\right)$
becomes a rough path, say $\mathbf{z}$. There is some ``tri-diagonal''
structure: (\ref{Transform1_psi}) can be solved as function of $z^{1}$
alone;
\begin{equation}
d\psi_{t}\left(x\right)=\sigma\left(t,\psi_{t}\left(x\right)\right)dz_{t}^{1}\text{ with }\psi_{0}\left(x\right)=x.\label{psi_rough}
\end{equation}
(\ref{Transform2_phi}) is tantamount to
\begin{equation}
\phi\left(t,r,x\right)=r\exp\left[\int_{0}^{t}\nu\left(s,\psi_{s}\left(x\right)\right)dz_{s}^{2}\right].\label{phi_rough}
\end{equation}
As for $\varphi=\varphi\left(t,x\right)$, note that
\[
1/\phi^{\prime}\left(t,r,x\right)=\tilde{\phi}\left(t,x\right)\equiv\exp\left[-\int_{0}^{t}\nu\left(s,\psi_{s}\left(x\right)\right)dz_{s}^{2}\right]
\]
so that
\begin{equation}
\varphi\left(t,x\right)=\int_{0}^{t}\tilde{\phi}\left(s,x\right)g\left(s,\psi_{s}\left(x\right)\right)dz_{s}^{3}.\label{a_rough}
\end{equation}

\begin{lemma} \label{LemmaRougFlows}Let $\mathbf{z}$ be a geometric
$p$-rough path; that is, an element in $C^{\text{0,}p\text{-var}}\left(\left[0,T\right],G^{\left[p\right]}\left(\mathbb{R}^{d}\right)\right)$.
Let $\gamma>p\geq1$. Assume
\begin{eqnarray*}
\sigma & = & \left(\sigma_{1},\dots,\sigma_{d_{1}}\right)\subset Lip^{\gamma}\left(\left[0,T\right]\times\mathbb{\mathbb{R}}^{n},\mathbb{\mathbb{R}}^{n}\right),\\
\nu & = & \left(\nu_{1},\dots,\nu_{d_{2}}\right)\subset Lip^{\gamma-1}\left(\left[0,T\right]\times\mathbb{\mathbb{R}}^{n},\mathbb{\mathbb{R}}\right),\\
g & = & \left(g_{1},\dots,g_{d_{3}}\right)\subset Lip^{\gamma-1}\left(\left[0,T\right]\times\mathbb{\mathbb{R}}^{n},\mathbb{\mathbb{R}}\right).
\end{eqnarray*}
Then $\psi,\phi$ and $\varphi$ depend (in local uniform sense) continuously
on $\left(z^{1},z^{2},z^{3}\right)$ in rough path sense. Under the
stronger regularity assumption $\gamma>p+2$; this also holds for
the first and second derivatives (with respect to $x$) of $\psi,\psi^{-1},\phi,\tilde{\phi}$
and $\varphi$. In particular, we can define $\psi,\phi$ and $\varphi$
when $\left(z^{1},z^{2},z^{3}\right)$ is replaced by a genuine geometric
$p$-rough path $\mathbf{z}$ and write $\psi^{\mathbf{z}},\phi^{\mathbf{z}},\varphi^{\mathbf{z}}$
to indicate this dependence. \end{lemma}

\begin{proof} Given $\mathbf{z}$ one can build a ``time-space''
rough path $\left(\mathbf{t,z}\right)$ of $\left(t,z^{1},z^{2},z^{3}\right)$
since the additionally needed iterated integrals against $t$ are
just Young integrals, cf.~\cite[Chapter 12]{friz-victoir-book}.
Define 
\[
W_{1}=\left(\begin{array}{l}
1\\
0\\
0\\
0\\
0
\end{array}\right),W_{2}=\left(\begin{array}{l}
0\\
\sigma\left(t,\psi\right)\\
0\\
0\\
0
\end{array}\right),W_{4}=\left(\begin{array}{l}
0\\
0\\
r.\phi.\nu\left(t,\psi\right)\\
-\tilde{\phi}.\nu\left(t,\psi\right)\\
0
\end{array}\right),W_{4}=\left(\begin{array}{l}
0\\
0\\
0\\
0\\
\tilde{\phi}.g\left(t,\psi\right)
\end{array}\right).
\]
The assumptions on $\sigma,\nu$ and $g$ guarantee that 
\[
W=\left(W_{1},\ldots,W_{4}\right):\mathbb{R}^{1+d_{1}+2d_{2}+d_{3}}\rightarrow L\left(\mathbb{R}^{1+d},\mathbb{R}^{1+d_{1}+2d_{2}+d_{3}}\right)
\]
is $Lip^{\gamma}$ (we work with $\mathbb{R}$ for the time coordinate
instead of the closed subset $\left[0,T\right]\subset\mathbb{R}$
since by the classic Whitney--Stein extension theorem (see e.g.~\cite{stein1970singular})
we can always find $Lip^{\gamma}$ resp.~$Lip^{\gamma-1}$ extensions
of $\sigma,\nu$ and $g$). Hence, the ``full RDE'' (parametrized
by $x\in\mathbb{R}^{n}$ and $r\in\mathbb{R}$)
\[
d\left(\begin{array}{l}
t\\
\psi\\
\phi\\
\tilde{\phi}\\
\varphi
\end{array}\right)=W\left(\begin{array}{l}
t\\
\psi\\
\phi\\
\tilde{\phi}\\
\varphi
\end{array}\right)d\left(\mathbf{t,z}\right)=\left(\begin{array}{llll}
1 & 0 & 0 & 0\\
0 & \sigma\left(t,\psi\right) & 0 & 0\\
0 & 0 & \phi.\nu\left(t,\psi\right) & 0\\
0 & 0 & -\tilde{\phi}.\nu\left(t,\psi\right) & 0\\
0 & 0 & 0 & \tilde{\phi}g\left(t,\psi\right)
\end{array}\right)d\left(\mathbf{t,z}\right)
\]
has a unique global solution%
\footnote{although $W$ fails to be bounded, the particular structure of the
system where one can first solve for $\psi$ and then construct the
other quantities by rough integration makes it clear that no explosion
can happen. The same situtation is discussed in detail in \cite[Chapter 11]{friz-victoir-book}.%
} (with obvious initial condition that the flows $\psi,\phi,\varphi$
evaluated at $t=0$ are the identity maps). Further, every additional
degree of Lipschitz regularity allows for one further degree of differentiability
of the solution flow with corresponding stability in rough path sense,
see \cite{lyons-98,lyons-qian-02,friz-victoir-book}. \end{proof}

\section{Semirelaxed limits and rough PDEs}

\label{sec:rpde}

The goal is to understand
\[
\partial_{t}u+L\left(t,x,u,Du,D^{2}u\right)=\sum_{i=1}^{d_{1}}\left(\sigma_{i}\left(t,x\right)\cdot Du\right)\dot{z}_{t}^{1;i}+u\sum_{j=1}^{d_{2}}\nu_{j}\left(t,x\right)\dot{z}_{t}^{2;j}+\sum_{k=1}^{d_{3}}g_{k}\left(t,x\right)\dot{z}_{t}^{3;k}
\]
in the case when $\left(z^{1},z^{2},z^{3}\right)$ becomes a rough
path. To this end we first give the assumptions on $L$.

\begin{assumption} \label{main_assumptions}Assume $L:\left[0,T\right]\times\mathbb{R}^{n}\times\mathbb{R}\times\mathbb{R}^{n}\times\mathbb{S}^{n}\rightarrow\mathbb{R}$
is of the form

\begin{equation}
L\left(t,x,r,p,X\right)=-\mathrm{Tr}\left[A(t,x)X\right]+b\left(t,x\right)\cdot p+c\left(t,x,r\right)\label{form_of_L_mainthm}
\end{equation}
with
\begin{enumerate}
\item $A=a\cdot a^{T}$ for some $a:\left[0,T\right]\times\mathbb{R}^{n}\rightarrow\mathbb{R}^{n\times n^{\prime}}$
\item $a:\left[0,T\right]\times\mathbb{R}^{n}\rightarrow\mathbb{R}^{n\times n^{\prime}}$
and $b:\left[0,T\right]\times\mathbb{R}^{n}\rightarrow\mathbb{R}^{n}$
are bounded, continuous in $t$ and Lipschitz continuous in $x$,
uniformly in $t\in\left[0,T\right]$
\item $c:\left[0,T\right]\times\mathbb{R}^{n}\times\mathbb{R}\rightarrow\mathbb{R}$
is continuous, bounded whenever $r$\ remains bounded, and with a
lower Lipschitz bound, i.e.\ $\exists C<0$ s.t.
\begin{equation}
c\left(t,x,r\right)-c\left(t,x,s\right)\geq C\left(r-s\right)\text{ for all }r\geq s,\text{ }\left(t,x\right)\in\left[0,T\right]\times\mathbb{R}^{n}\text{.}\label{EqWeakProper}
\end{equation}

\end{enumerate}
\end{assumption}

Assumption \ref{main_assumptions} guarantees that a comparison result
holds for $\partial_{t}+L$; see the appendix and \cite[Section V, Lemma 7.1]{MR2179357FS}
or \cite{MR2765508} for details. Further we need the assumptions
on the coefficients in $\Lambda$.

\begin{assumption}\label{as:lambda}Assume that%
\footnote{The regularity assumptions on the vector fields with respect to $t$
could be relaxed here, cf.~remark \ref{rem.timereg}. %
}

\begin{eqnarray*}
\sigma & = & \left(\sigma_{1},\dots,\sigma_{d_{1}}\right)\subset Lip^{\gamma}\left(\left[0,T\right]\times\mathbb{\mathbb{R}}^{n},\mathbb{\mathbb{R}}^{n}\right),\\
\nu & = & \left(\nu_{1},\dots,\nu_{d_{2}}\right)\subset Lip^{\gamma-1}\left(\left[0,T\right]\times\mathbb{\mathbb{R}}^{n},\mathbb{\mathbb{R}}\right),\\
g & = & \left(g_{1},\dots,g_{d_{3}}\right)\subset Lip^{\gamma-1}\left(\left[0,T\right]\times\mathbb{\mathbb{R}}^{n},\mathbb{\mathbb{R}}\right).
\end{eqnarray*}

\end{assumption}

Let us now replace the (smooth)\ driving signals of the earlier sections
by a $d=\left(d_{1}+d_{2}+d_{3}\right)$-dimensional driving signal
$z^{\varepsilon}$ and \textit{impose} convergence to a genuine geometric
$p$-rough path $\mathbf{z}$, that is, in the notation of \cite[Chapter 14]{friz-victoir-book}
\[
\mathbf{z}\in C^{0,p\text{-var}}\left(\left[0,T\right],G^{\left[p\right]}\left(\mathbb{R}^{d_{1}+d_{2}+d_{3}}\right)\right).
\]
We can now prove our main result.

\begin{theorem} \label{ThmMain}Let $p\geq1$. Assume $L$ fulfills
assumption \ref{main_assumptions} and the coefficients of $\Lambda=\left(\Lambda_{1},\dots,\Lambda_{d_{1}+d_{2}+d_{3}}\right)$
fulfill \ref{as:lambda} for some $\gamma>p+2.$ Let $u_{0}\in BUC\left(\mathbb{R}^{n}\right)$
and $\mathbf{z}\in C^{0,p\text{-var}}\left(\left[0,T\right],G^{\left[p\right]}\left(\mathbb{R}^{d}\right)\right)$.
Then there exists a unique $u=u^{\mathbf{z}}\in BUC\left(\left[0,T\right]\times\mathbb{R}^{n}\right)$
such that for any sequence $\left(z^{\epsilon}\right)_{\epsilon}\subset C^{1}\left(\left[0,T\right],\mathbb{R}^{d}\right)$
such that $z^{\varepsilon}\rightarrow\mathbf{z}$ in $p$-rough path
sense, the viscosity solutions $\left(u^{\varepsilon}\right)\subset BUC\left(\left[0,T\right]\times\mathbb{R}^{n}\right)$
of
\[
\dot{u}^{\varepsilon}+L\left(t,x,u^{\varepsilon},Du^{\varepsilon},D^{2}u^{\varepsilon}\right)=\sum_{k=1}^{d}\Lambda_{k}\left(t,x,u^{\varepsilon},Du^{\varepsilon}\right)\dot{z}^{k;\varepsilon}\mathbf{,\,\,\,}u^{\epsilon}\left(0,\cdot\right)=u_{0}\left(.\right),
\]
converge locally uniformly to $u^{\mathbf{z}}$. We write formally%
\footnote{The intrinsic meaning of this ``rough''\ PDE is discussed in definition
\ref{roughPDE} below.%
},
\begin{equation}
du+L\left(t,x,u,Du,D^{2}u\right)dt=\Lambda\left(t,x,u,Du\right)d\mathbf{z,\,\,\,}u\left(0,\cdot\right)=u_{0}\left(.\right)\label{EqRPDE}
\end{equation}
Moreover, we have the contraction property
\[
\left\vert u^{\mathbf{z}}-\hat{u}^{\mathbf{z}}\right\vert _{\infty;\mathbb{R}^{n}\times\left[0,T\right]}\leq e^{CT}\left\vert u_{0}-\hat{u}_{0}\right\vert _{\infty;\mathbb{R}^{n}}
\]
($C$ given by \eqref{EqWeakProper}) and continuity of the solution
map $\left(\mathbf{z,}u_{0}\right)\mapsto u^{\mathbf{z}}$ from 
\[
C^{0,p\text{-var}}\left(\left[0,T\right],G^{\left[p\right]}\left(\mathbb{R}^{d}\right)\right)\times\mathrm{BUC}\left(\mathbb{R}^{n}\right)\rightarrow\mathrm{BUC}\left(\left[0,T\right]\times\mathbb{R}^{n}\right).
\]
\end{theorem}

\begin{proof} We shall write $\psi^{\mathbf{z}},\phi^{\mathbf{z}},\varphi^{\mathbf{z}}$
for the objects (solutions of rough differential equations and integrals)
built upon $\mathbf{z,}$ as discussed in the last section (lemma
\ref{LemmaRougFlows}) and also write $\psi^{\mathbf{\varepsilon}},\phi^{\mathbf{\varepsilon}},\varphi^{\mathbf{\varepsilon}}$
when the driving signal is $z^{\varepsilon}$. Recall from (\ref{DefMk})
that $\Lambda=\left(\Lambda_{1},\dots,\Lambda_{d}\right)$ is a collection
of linear, first order differential operators. We use the same technique
of ``rough semi-relaxed limits'' as in \cite{MR2765508}: the key
remark being that
\[
\left[\,^{\phi^{\varepsilon}}\left(L^{\psi^{\varepsilon}}\right)\right]_{\varphi^{\varepsilon}}\rightarrow\left[\,^{\phi^{\mathbf{z}}}\left(L^{\psi^{\mathbf{z}}}\right)\right]_{\varphi^{\mathbf{z}}}
\]
holds locally uniformly, as function of $\left(t,x,r,p,X\right)$.
Secondly, applying the transformations, the (classical) viscosity
solutions $u^{\varepsilon}$ can be used to define a new function
$\tilde{u}^{\varepsilon}$ by setting 
\[
\tilde{u}^{\varepsilon}\left(t,x\right)=\left(\phi^{\varepsilon}\right)^{-1}\left(t,u^{\varepsilon}\left(t,\psi^{\varepsilon}\left(t,x\right)\right),x\right)+\varphi^{\varepsilon}\left(t,x\right);
\]
and proposition \ref{Prop_Transform_to_PDE} in section \ref{sec:trans}
show that $\tilde{u}^{\varepsilon}$ is a (classical) viscosity solution
of 
\[
d\tilde{u}^{\varepsilon}+\left[\,^{\phi^{\varepsilon}}\left(L^{\psi^{\varepsilon}}\right)\right]_{\varphi^{\varepsilon}}\left(t,x,\tilde{u}^{\varepsilon},D\tilde{u}^{\varepsilon},D^{2}\tilde{u}^{\varepsilon}\right)=0.
\]
Now one uses the comparison result for parabolic viscosity solutions
(as given in the appendix) to conclude that there exists a constant
$C>0$ such that
\[
\sup_{\substack{\varepsilon\in(0,1]\\
t\in\left[0,T\right]\\
x\in\mathbb{R}^{n}
}
}\left\vert \tilde{u}^{\varepsilon}\left(t,x\right)\right\vert <\left(1+\left\vert u_{0}\right\vert _{\infty}\right)e^{CT};
\]
This in turn implies (thanks to the uniform control on $\varphi^{\varepsilon},\phi^{\varepsilon},\psi^{\varepsilon}$
as $\varepsilon\rightarrow0$) by using the rough path representations
discussed in section \ref{SectionRT} that $\tilde{u}^{\varepsilon}$
remains locally uniform bounded (as $\varepsilon\rightarrow0$). Together
with the stability of (classical) viscosity solutions (c.f.~\cite{MR2765508})
the proof is finished. \end{proof}

The reader may wonder if $u$ is the solution in a sense beyond the
``formal'' equation
\[
du+L\left(t,x,u,Du,D^{2}u\right)dt=\Lambda\left(t,x,u,Du\right)d\mathbf{z,\,\,\,}u\left(0,\cdot\right)\equiv u_{0}\left(\cdot\right).
\]
Inspired by the definition given by Lions--Souganidis in \cite{MR1799099}
we give

\begin{definition} \label{roughPDE}$u$ is a solution to the \textbf{rough
partial differential equation} \eqref{EqRPDE} if and only if $\tilde{u}\left(t,x\right)=\left(\phi^{\boldsymbol{z}}\right)^{-1}\left(t,u\left(t,\psi^{\boldsymbol{z}}\left(t,x\right)\right),x\right)+\varphi^{\boldsymbol{z}}\left(t,x\right)$
\[
d\tilde{u}+\tilde{L}\left(t,x,\tilde{u},D\tilde{u},D^{2}\tilde{u}\right)=0,\,\,\,\tilde{u}\left(0,\cdot\right)=u_{0}\left(\cdot\right)
\]
in viscosity sense where
\[
\tilde{L}=\,\left[^{\phi^{\mathbf{z}}}\left(L^{\psi^{\mathbf{z}}}\right)\right]_{\varphi^{\mathbf{z}}}.
\]
\end{definition}

\begin{corollary} Under the assumptions of Theorem \ref{ThmMain}
there exists a unique solution in $BUC\left(\left[0,T\right]\times\mathbb{R}^{n}\right)$
to the RPDE \eqref{EqRPDE}. \end{corollary}

\begin{proof} Existence is clear from theorem \ref{ThmMain}. Uniqueness
is inherited from uniqueness to the Cauchy problem for $\left(\partial_{t}+\tilde{L}\right)=0$
which follows from a comparison theorem for parabolic viscosity solution
(c.f.~Theorem \ref{thm:comp} in the appendix). \end{proof}

\section{\label{sec:viscL2}RPDEs and Variational Solutions of SPDEs}

A classic approach to second order parabolic SPDEs (especially the
Zakai equation from nonlinear filtering) is the so-called $L^{2}$-theory
for SPDEs, due to Pardoux, Rozovskii, Krylov et.~al., cf.~\cite{MR557763,MR553909,MR2329435}.
Sufficient conditions for existence, uniqueness in this setting are
classical (brief recalls are given below). On the other hand, our
main theorem on RPDEs driven by rough paths, theorem \ref{ThmMain},
can be applied with almost every realization of Enhanced Brownian
motion, $\boldsymbol{B}(\omega)$, that is Brownian motion enhanced
with Lévy's stochastic area, the standard example of a (random) rough
path. The aim of this section is to show that, whenever the (not too
far from optimal) assumptions of both theories are met, the resulting
solutions coincide. We focus on the model case of linear SPDEs; i.e.
\begin{equation}
L\left(t,x,r,p,X\right)=-\mathrm{Tr}\left[A(t,x)X\right]+b\left(t,x\right)\cdot p+c\left(t,x\right)r+f\left(t,x\right).\label{eq:L linear}
\end{equation}

\subsection{$L^{2}$ solutions}

Given is a filtered probability space $\left(\Omega,\mathcal{F},\left(\mathcal{F}_{t}\right),\mathbb{P}\right)$,
which satisfies the usual conditions and carries a $d$-dimensional
Brownian motion $B$. Denote with $H^{m}\left(\mathbb{R}^{n}\right)$
the usual Sobolev space, i.e.~the subspace of $L^{p}\left(\mathbb{R}^{n}\right)$
consisting of functions whose generalized derivatives up to order
$m$ are in $L^{p}\left(\mathbb{R}^{n}\right)$. Equipped with the
norm 
\[
\left|f\right|_{H^{m}\left(\mathbb{R}^{n}\right)}=\left(\sum_{0\leq k\leq m}\left|\partial_{i_{1}}\cdots\partial_{i_{k}}f\right|_{L^{2}\left(\mathbb{R}^{n}\right)}^{p}\right)^{1/2}
\]
$H^{m}\left(\mathbb{R}^{n}\right)$ becomes a separable Hilbert space
and the variational approach makes use of the triple 
\[
\left(H^{m}\left(\mathbb{R}^{n}\right)\right)\hookrightarrow L_{r}^{2}\left(\mathbb{R}^{n}\right)\backsimeq\left(L_{r}^{2}\left(\mathbb{R}^{n}\right)\right)^{\star}\hookrightarrow\left(H^{m}\left(\mathbb{R}^{n}\right)\right)^{\star}.
\]

We make the following assumptions on the coefficients of $L$ and
$\Lambda$.

\begin{assumption} \label{assumptions} For $i,j\in\left\{ 1,\ldots,n\right\} $,
$k\in\left\{ 1,\ldots,d\right\} $
\begin{enumerate}
\item $a^{ij},b_{i},c,f$ as well as $\sigma_{k}^{i},\nu_{k}^{i},g$ are
elements of%
\footnote{$C_{b}$ denotes the bounded continuous functions and $C_{0}$ the
subset of compactly supported functions. %
} $C_{b}\left(\left[0,T\right]\times\mathbb{R}^{n},\mathbb{R}\right)$
and $\sigma_{k}^{i},\nu_{k}^{i},g$ have one and $a^{ij},\sigma_{k}^{i}$
have two, bounded (uniformly in $t$) continuous derivatives in space,
\item $f,g\in L^{2}\left(\left[0,T\right]\times\mathbb{R}^{n}\right)$,
\item $A=a\cdot a^{T}$, and $\exists\lambda>0$ such that $\forall t\in\left[0,T\right]$
\[
z^{T}\cdot A\left(t,x\right)\cdot z\geq\lambda\left|z\right|^{2}\,\,\,\forall x,z\in\mathbb{R}^{n}.
\]

\end{enumerate}
\end{assumption}

A $L^{2}$-solution is then defined as follows

\begin{definition} \label{def: L2 solution}Let $u_{0}\in L^{2}\left(\mathbb{R}^{n}\right)$
and assume $L$ and $\Lambda$ fulfill assumption \ref{assumptions}.
We say that a $L^{2}\left(\mathbb{R}^{n}\right)$-valued, strongly
continuous and $\left(\mathcal{F}_{t}\right)$-adapted process $u=\left(u_{t}\right)_{t\in\left[0,T\right]}$
is a $L^{2}$-solution of 
\begin{eqnarray}
du & = & Ludt+\Lambda u\circ dB_{t}\label{eq:spde}\\
u\left(0,.\right) & = & u_{0}\left(.\right)\nonumber 
\end{eqnarray}
if 
\begin{enumerate}
\item we have $\mathbb{P}$-a.s.~that $u_{t}\in H^{1}\left(\mathbb{R}^{n}\right)$
for a.e.~$t\in\left[0,T\right]$ and $\mathbb{P}\left(\int_{0}^{T}\left|u_{r}\right|_{H^{1}\left(\mathbb{R}^{n}\right)}^{2}dr<\infty\right)=1$,
\item $\forall\varphi\in C_{0}^{\infty}\left(\mathbb{R}^{n}\right)$ we
have%
\footnote{$\widetilde{L}^{\star}$ and $\Lambda^{\star}$ denote the formal
adjoint operators of $\widetilde{L}$ and $\Lambda$, the stochastic
integral is understood in the Ito sense and $\left\langle .,.\right\rangle $
denotes the scalar product on $L^{2}\left(\mathbb{R}^{n}\right)$%
} 
\begin{equation}
\left\langle u_{.},\varphi\right\rangle -\left\langle u_{0},\varphi\right\rangle =\int_{0}^{.}\left\langle u_{r},\widetilde{L}^{\star}\varphi\right\rangle dr+\sum_{k=1}^{d}\int_{0}^{.}\left\langle u_{r},\Lambda_{k}^{\star}\varphi\right\rangle dB_{r}^{k}\,\,\,\left(d\lambda\otimes d\mathbb{P}\right)-a.s.,\label{eq:weak_non_div}
\end{equation}
 here $\widetilde{L}\varphi:=L\varphi+\frac{1}{2}\sum_{k=1}^{d}\Lambda_{k}\Lambda_{k}\varphi$.
\end{enumerate}
\end{definition} 

\begin{remark}The difference with the standard definition, cf.~\cite[Chapter IV, p130]{MR1135324},
is that we additionally assume enough regularity on the coefficients
for the existence of the adjoint of $\widetilde{L}$ and to switch
between divergence and non-divergence form%
\footnote{ In the classic variational approach this can be avoided by working
throughout with $L$ in divergence form (resulting in no smoothness
requirement on the coefficients in space instead of the existence
of one derivative; in fact, except for the free terms, only boundedness
and measurability of coefficients in combination with superparabolicity
is sufficient, cf.~\cite[Chapter IV]{MR1135324}).%
}.\end{remark}

\begin{theorem}\label{thm:L2 weak}Under assumption \ref{assumptions}
there exists a unique $L^{2}$-solution of (\ref{eq:spde}).\end{theorem}
\begin{proof}The standard variational approach as for example presented
in \cite[Chapter 4, Theorem 1]{MR1135324} (see also~\cite{MR557763,MR553909,MR2329435})
guarantees the existence of an $L^{2}\left(\mathbb{R}^{n}\right)$-valued,
strongly continuous in $t$, $\left(\mathcal{F}_{t}\right)$-adapted
process $\left(u_{t}\right)$ which fulfills part (2) of definition
\ref{def: L2 solution} as well as that $\forall\varphi\in C_{0}^{\infty}\left(\mathbb{R}^{n}\right)$
we have $\mathbb{P}$-a.s (using the Einstein summation convention%
\footnote{for $i,j\in\left\{ 1,\ldots,n\right\} $ and $k\in\left\{ 1,\ldots,d_{1},\ldots,d_{1}+d_{2},\ldots,d_{1}+d_{2}+d_{3}\right\} $
and setting $\sigma_{k}=0$ for $k>d_{1}$,$\nu_{k}=0$ for $k\leq d_{1}$
or $k>d_{1}+d_{2}$ and $g_{k}=0$ for $k\leq d_{1}+d_{2}$%
}) 
\begin{align*}
\left\langle u_{t},\varphi\right\rangle -\left\langle u_{0},\varphi\right\rangle  & =\int_{0}^{t}\left(-\left\langle \tilde{a}^{ij}\partial_{j}u_{r},\partial_{i}\varphi\right\rangle +\left\langle \overline{b}_{r}^{i}\partial_{i}u_{r}+\tilde{c}_{r}u_{r}+\tilde{f}_{r},\varphi\right\rangle \right)dr\\
 & \,\,\,\,+\int_{0}^{t}\left\langle \sigma_{k}^{j}\left(r\right)\partial_{j}u_{r}+\nu_{k}\left(r\right)u_{r}+g_{k}\left(r\right),\varphi\right\rangle dB_{r}^{k}
\end{align*}
where 
\begin{align*}
\tilde{a}^{ij} & =a^{ij}+\frac{1}{2}\sum_{k=1}^{d}\sigma_{k}^{i}\sigma_{k}^{j}\\
\tilde{b}^{i} & =b^{i}+\frac{1}{2}\sum_{k=1}^{d}\left(\sigma_{k}^{j}\left(\partial_{j}\sigma_{k}^{i}\right)+2\nu_{k}\sigma_{k}^{i}\right)\\
\tilde{c} & =c+\frac{1}{2}\sum_{k=1}^{d}\left(\sigma_{k}^{i}\left(\partial_{i}\nu_{k}\right)+\nu_{k}^{2}\right)\\
\tilde{f} & =f+\frac{1}{2}\sum_{k=1}^{d}\left(\sigma_{k}^{i}\partial_{i}g+\nu_{k}g_{k}+g\right)
\end{align*}
and 
\[
\overline{b}^{i}=\tilde{b}^{i}-\left(\partial_{j}\tilde{a}^{ij}\right)
\]
(i.e.~$\tilde{a}$ and $\tilde{b}$ are the coefficients that appear
in $\widetilde{L}$ due to the switch from Ito to Stratonovich integration
and $\overline{b}$ results from the switch to divergence form). Now
using integration by parts we can rewrite the above divergence form
into the adjoint formulation \eqref{eq:weak_non_div} as required
by definition \ref{def: L2 solution}.\end{proof}We can now prove
the main result of this section which identifies the RPDE solution
with the classic $L^{2}$-solution whenever both exist.

\begin{proposition} \label{prop:weak L2}Assume that $L$ and $\Lambda$
fulfill assumption \ref{assumptions}\textbf{ }as well as the assumptions
of Theorem \ref{ThmMain}. If we denote with $u^{\boldsymbol{B}}$
the RPDE solution given in Theorem \ref{ThmMain} driven by Enhanced
Brownian motion $\boldsymbol{B}$ then $\left(u_{t}^{\boldsymbol{B}}\right)_{t\geq0}$
is a (and hence a version of the unique) $L^{2}$-solution.\end{proposition}

\begin{proof} \textbf{Step 1.} Assume additionally to assumption
\ref{assumptions} that all coefficients appearing in $L$ and $\Lambda$
are in $C_{0}^{\infty}\left(\left(0,T\right)\times\mathbb{R}^{n}\right)$
(in step 2 we get rid of this assumption). Define the adapted, piecewise
linear approximation $B^{n}$ to $B$ as 
\[
B_{t}^{n}=B_{t_{k-1}}+n\frac{\left(t-t_{k-1}\right)}{T}\left(B_{t_{k}}-B_{t_{k-1}}\right)
\]
for $t\in\left[t_{k},t_{k+1}\right)$ with $t_{k}=k\frac{T}{n}$.
For every $n\in\mathbb{N}$ we denote with $u\left(B_{n}\right)$
the $L^{2}$-solutions of (\ref{eq:spde}) where $B$ is replaced
by $B^{n}$ and with $u$ the unique $L^{2}$-solution of theorem
\ref{thm:L2 weak}. Further, denote with $\boldsymbol{B}^{n}=S_{2}\left(B^{n}\right)$
the rough path lift of $B^{n}$ and with $u^{\boldsymbol{B}_{n}}$
resp.~$u^{\boldsymbol{B}}$ the viscosity solution for the random
rough path $\boldsymbol{B}^{n}$ resp.~$\boldsymbol{B}$ given in
theorem \ref{ThmMain}. For $\epsilon>0$ and $n\in\mathbb{N}$ the
regularity of the coefficients allows to identify $u^{\boldsymbol{B}_{n}\left(\omega\right)}$
with $u\left(B^{n}\left(\omega\right)\right)$ (both are the unique,
bounded, smooth solutions of a parabolic PDE with smooth coefficients
which depend on $\omega$). The Wong--Zakai result in \cite[Theorem 2.1]{MR2268661}
(all coefficients are smooth and additionally $f\left(t,.\right),g\left(t,.\right)$
have compact support, hence are $H^{5}\left(\mathbb{R}^{n}\right)$-valued),
tells us that 
\[
u\left(B_{n}\right)\rightarrow u\left(B\right)\in L^{2}\left(\left[0,T\right],H^{1}\right)\text{ a.s.}
\]
where the convergence takes place in $L^{2}\left(\left[0,T\right],H^{1}\right)$
and hence also in $L^{2}\left(\left[0,T\right]\times\mathbb{R}^{n}\right)$
(\textit{much} more is true here of course). On the other hand, we
know that a.s.~$\boldsymbol{B}^{n}$ converges to $\boldsymbol{B}$
in rough path metric (see \cite{friz-05}) and from \ref{ThmMain}
we conclude that 
\[
u^{\boldsymbol{B}_{n}}\rightarrow u^{\mathbf{B}}\text{ a.s.}
\]
locally uniformly on $\left[0,T\right]\times\mathbb{R}^{n}$. It is
an now easy matter to identify the $L^{2}$ and loc.~uniform limit,
\[
u\left(B\right)=u^{\mathbf{B}}\text{ a.s}
\]
(viewing $u\left(B\right)$ a.s.~$C\left(\left[0,T\right],L^{2}\left(\mathbb{R}^{n}\right)\right)$-valued,
this means that for a.e.~$\omega$, $\forall t\in\left[0,T\right]$,
$u_{t}\left(B\left(\omega\right)\right)=u_{t}^{\mathbf{B}\left(\omega\right)}$
as equality in $L^{2}\left(\mathbb{R}^{n}\right)$; in particular,
$u^{\mathbf{B}\left(\omega\right)}$ constitutes a continuous version
in $t,x$; once more much more is true here). Further we know that
$u\left(B\right)\in L^{2}\left(\left[0,T\right],H^{1}\left(\mathbb{R}^{n}\right)\right)$
$\mathbb{P}$-a.s.~and hence fulfills definition \ref{def: L2 solution}.
Hence, we conclude that $\left(u_{t}^{\mathbf{B}}\right)$ is the
unique $L^{2}$-solution of (\ref{eq:spde}) (strictly speaking, a
continuous function in the equivalence class).

\textbf{Step 2.} For every $\epsilon>0$ truncate all coefficients
outside a ball of radius $\epsilon^{-1}$ and smooth by convolution
with a mollifier $m^{\epsilon}$ (viz.~$m^{\epsilon}\left(t,x\right)=\epsilon^{-n+1}m\left(\frac{t}{\epsilon},\frac{x}{\epsilon}\right)$
where $m$ has compact support, is non-negative and has total mass
one)%
\footnote{\label{fn:mol}The mollification uses values of the coefficients for
$t$ outside $\left[0,T\right]$ therefore we simply define the coefficients
for $t\in\mathbb{R}\backslash\left[0,T\right]$ by constant continuation.%
} to arrive at the new operators $L^{\epsilon}$ and $\Lambda^{\epsilon}$.
It is easy to see that $L^{\epsilon}$,$\Lambda^{\epsilon}$ again
fulfill assumption \ref{assumptions}, hence Theorem \ref{thm:L2 weak}
applies and gives existence and uniqueness of the associated $L^{2}$-solution.
Denote with $u^{\epsilon;\boldsymbol{B}}$ the associated RPDE solution%
\footnote{with abuse of notation we identify the operators $L^{\epsilon}$,
$L$ (and similarly $\Lambda$,$\Lambda^{\epsilon}$) with functions
on $\left[0,T\right]\times\mathbb{R}^{n}\times\mathbb{R}\times\mathbb{R}^{n}\times\mathbb{S}^{n}\rightarrow\mathbb{R}$
as required in the viscosity setting in the obvious way.%
} (with random rough path $\boldsymbol{B}$) and note that by step
1, $u^{\epsilon;\boldsymbol{B}}$ coincides with the $L^{2}$-solution.
We now claim that 
\[
u^{\epsilon,\boldsymbol{B}}\rightarrow u^{\boldsymbol{B}}\,\,\, a.s.
\]
uniformly on compacts. To see this, note that by the construction
given in theorem \ref{ThmMain}, $u^{\boldsymbol{B}}$ is the composition
of a viscosity solution $\tilde{u}^{\epsilon}$ with rough path flows
and $\tilde{u}^{\epsilon}$ itself is a solution of a linear PDE 
\[
\partial v_{t}^{\epsilon}+\overline{L}^{\epsilon}\left(t,x,v^{\epsilon},Dv^{\epsilon},D^{2}v^{\epsilon}\right)=0;
\]
the precise form of $\overline{L}^{\epsilon}$ is as in Theorem \ref{ThmMain}
given by the transformation via rough path flows, that is 
\[
\overline{L}^{\epsilon}=\left[^{\phi^{\boldsymbol{B},\epsilon}}\left(\left(L^{\epsilon}\right)^{\psi^{\boldsymbol{B},\epsilon}}\right)\right]_{\varphi^{\boldsymbol{B},\epsilon}}.
\]
(where $\phi^{\boldsymbol{B},\epsilon},\psi^{\boldsymbol{B},\epsilon}$
and $\varphi^{\boldsymbol{B},\epsilon}$ denote the rough flows associated
with the truncated and mollified vector fields). Further we claim
that the truncated and mollified $Lip^{\gamma}$-vector fields (appearing
in $\Lambda^{\epsilon}$) converge locally uniformly with locally
uniform $Lip^{\gamma}$ bounds: given $V\in Lip^{\gamma}$ denote
$V^{\epsilon}$ as the vector field given by truncation outside radius
$\epsilon^{-1}$ and convolution of $V$ with $m^{\epsilon}$. Of
course, $V^{\epsilon}$ converges locally uniformly to $V$ and (cf.~\cite[p123, p159]{stein1970singular})
locally uniform $Lip^{\gamma}$ bounds are readily seen to hold true
for every $V^{\epsilon}$, $\epsilon>0$. An interpolation argument
then shows, locally, convergence in $Lip^{\gamma'}$ for $\gamma'<\gamma$
(we only need $\gamma'=\gamma-1$). Given a geometric $p$-rough path
$\boldsymbol{z}$ with $p<\gamma$ it then follows from \cite[Corollary 10.39]{friz-victoir-book}
(together with the remark that the $\left|.\right|_{Lip^{\gamma-1}}$-norm
can be replaced by the local Lipschitz norm, restricted to a big enough
ball in which both RDE solutions live) that the (unique) RDE solutions
(started at a fixed point) to $dy^{\epsilon}=V^{\epsilon}\left(y^{\epsilon}\right)d\boldsymbol{\boldsymbol{z}}$
converge as $\epsilon\rightarrow0$ to the (unique) RDE solution of
$dy=V\left(y\right)d\boldsymbol{\boldsymbol{z}}$. As in \cite[Theorem 11.12 and Theorem 11.13]{friz-victoir-book}
this convergence can be raised to the level of $C^{k}$--diffeomorphisms,
provided $V$ is assumed to be $Lip^{\gamma+k-1}$ for $k\in\mathbb{N}$
-- the case of interest to us is given by $\gamma>4$ and $p\in\left(2,\gamma-2\right)$
which results in stability of the flow seen as $C^{2}$--diffeomorphism.
This shows that $\overline{L}^{\epsilon}\rightarrow\left[^{\phi^{\boldsymbol{B}}}\left(L^{\psi^{\boldsymbol{B}}}\right)\right]_{\varphi^{\boldsymbol{B}}}$
as $\epsilon\rightarrow0$ uniformly on compacts and the stability
properties of viscosity solutions guarantee (the same argument as
given in theorem \ref{ThmMain}) that $v^{\epsilon}\rightarrow v$,
hence $u^{\boldsymbol{B},\epsilon}\rightarrow u^{\boldsymbol{B}}$
(loc.~uniformly on $\left[0,T\right)\times\mathbb{R}^{n}$) a.s.
From the first step it follows that $u^{\boldsymbol{B},\epsilon}$
is the unique $L^{2}$-solution, i.e. $\forall t\in\left[0,T\right]$
\[
\left\langle u_{t}^{\boldsymbol{B},\epsilon},\varphi\right\rangle -\left\langle u_{0},\varphi\right\rangle =\int_{0}^{t}\left\langle u_{t}^{\boldsymbol{B},\epsilon},\left(\widetilde{L}^{\epsilon}\right)^{\star}\varphi\right\rangle dr+\sum_{k=1}^{d}\int_{0}^{t}\left\langle u_{r}^{\boldsymbol{B},\epsilon},\left(\Lambda_{k}^{\epsilon}\right)^{\star}\varphi\right\rangle dB_{r}^{k}
\]
Sending $\epsilon\rightarrow0$ in above equality shows that point
(2) of definition \ref{def: L2 solution} is fulfilled. Now for every
$\epsilon>0$, classic variational arguments, see \cite[Chapter 4, Theorem 1, p131]{MR1135324},
show that there exists a constant $C^{\epsilon}>0$ which depends
only on $T,n,d$ and $\sup_{t,x,i,j,k}\left(\left|\tilde{a}^{\epsilon;ij}\right|,\left|b^{\epsilon;i}\right|,\left|\sigma_{k}^{\epsilon;i}\right|,\left|\nu_{k}^{\epsilon;i}\right|\right)$
(which is finite by assumption \ref{assumptions}) s.t.~ 
\begin{align*}
 & \mathbb{E}\left[\sup_{t\in\left[0,T\right]}\left|u_{r}^{\epsilon,\boldsymbol{B}}\right|_{L^{2}\left(\mathbb{R}^{n}\right)}^{2}+\int_{0}^{T}\left|u_{r}^{\epsilon,\boldsymbol{B}}\right|_{H^{1}\left(\mathbb{R}^{n}\right)}^{2}dr\right]\\
\leq & C^{\epsilon}.\left[\left|u_{0}\right|_{L^{2}\left(\mathbb{R}^{n}\right)}^{2}+\mathbb{E}\int_{0}^{T}\left(\left|f_{r}^{\epsilon}\right|_{H^{-1}\left(\mathbb{R}^{n}\right)}^{2}+\sum_{k=d_{1}+d_{2}+1}^{d=d_{1}+d_{2}+d_{3}}\left|\left(g_{r}^{\epsilon}\right)^{k}\right|_{L^{2}\left(\mathbb{R}^{n}\right)}^{2}\right)dr\right].
\end{align*}
By the estimate 
\[
\left|f_{r}^{\epsilon}\right|_{H^{-1}\left(\mathbb{R}^{n}\right)}\lesssim\left|f_{r}^{\epsilon}\right|_{L^{2}\left(\mathbb{R}^{n}\right)}=\left|\left(f\left(.\right)1_{\left|.\right|<\epsilon^{-1}}\right)\star m^{\epsilon}\right|_{L^{2}\left(\mathbb{R}^{n}\right)}\leq\left|f\left(.\right)1_{\left|.\right|<\epsilon^{-1}}\right|_{L^{2}\left(\mathbb{R}^{n}\right)}\leq\left|f\right|_{L^{2}\left(\mathbb{R}^{n}\right)}
\]
(and similarly for $g^{\epsilon}$), the right-hand side can be uniformly
bounded in $\epsilon$, leading to the desired regularity properties
of $u^{\boldsymbol{B}}$ (as required by point (2) in definition \ref{def: L2 solution}).\end{proof}

\begin{remark} Classical $L^{2}$-theory of SPDEs gives, with probability
one, $u(t,.,\omega)\in L^{2}\left(\mathbb{R}^{n}\right)$ for all
$t\in[0,T]$ and then in the Sobolev space $H^{1}\left(\mathbb{R}^{n}\right)$
for a.e. $t\in[0,T]$. It is not clear, in general, if a continuous
(in $t,x$) version of $u$ exists. Under further regularity assumptions
one finds that $u(t,.,\omega$) takes values in higher Sobolev spaces
$H^{l}\left(\mathbb{R}^{n}\right),l=1,2...$. Since Sobolev embedding
theorems are dimension-dependent (recall $H^{l}\left(\mathbb{R}^{n}\right)\subset C\left(\mathbb{R}^{n}\right)$
when $l>n/2$) the regularity required for a continuous version will
grow with the dimension $n$. In contrast, our approach effectively
gives sufficient conditions, without dimension dependence, under which
$L^{2}$-solutions admit continuous versions. We note that such considerations
also motivated Krylov's $L^{p}$-theory \cite[p185]{KrylovAnAp}\end{remark}

\subsection{A $L_{loc}^{1}$-solution}

Theorem \ref{ThmMain} applied with enhanced Brownian motion provides
the unique RPDE viscosity solution even if 
\begin{enumerate}
\item $L$ is degenerate elliptic,
\item $u_{0}\in BUC\left(\mathbb{R}^{n}\right)$. 
\end{enumerate}
Under such conditions one can not hope for the existence of a $L^{2}$-solution:
the degeneracy of $L$ does not lead to $H^{1}$-regularity in space
and the initial data $u_{0}$ does not fit into a $L^{2}$-theory
(in fact $L^{p}$ for $1\leq p<\infty$, e.g.~by taking~$u_{0}\equiv1$;
however one could consider weighted Sobolev spaces). Hence, our new
assumptions read,

\begin{assumption} \label{assumptions-L_loc} For $i,j\in\left\{ 1,\ldots,n\right\} $,
$k\in\left\{ 1,\ldots,d\right\} $,
\begin{enumerate}
\item $a^{ij},b_{i},c,f$ as well as $\sigma_{k}^{i},\nu_{k},g_{k}$ are
in $C_{b}\left(\left[0,T\right]\times\mathbb{R}^{n}\right)$ and $\sigma_{k}^{i},\nu_{k}^{i},g$
have one and $a^{ij},\sigma_{k}^{i}$ have two, continuous, bounded
(uniformly in $t$) derivatives in space,
\item $f,g\in L^{2}\left(\left[0,T\right]\times\mathbb{R}^{n}\right)$,
\item $A=a\cdot a^{T}$.
\end{enumerate}
\end{assumption}Motivated by above remarks we give the following
definition. \begin{definition} By an $L_{loc}^{1}$-solution we mean
a $L_{loc}^{1}\left(\mathbb{R}^{n}\right)$-valued strongly continuous
$\left(\mathcal{F}_{t}\right)$-adapted process $u=\left(u_{t}\right)_{t\in\left[0,T\right]}$
such that $\forall\varphi\in C_{0}^{\infty}\left(\mathbb{R}^{n}\right)$
we have 
\[
\left\langle u_{.},\varphi\right\rangle -\left\langle u_{0},\varphi\right\rangle =\int_{0}^{.}\left\langle u_{r},\left(\widetilde{L}\right)^{\star}\varphi\right\rangle dr+\sum_{k=1}^{d}\int_{0}^{.}\left(u,\left(\Lambda_{k}\right)^{\star}\varphi\right)_{r}dB_{r}^{k}\,\,\,\left(d\lambda\otimes d\mathbb{P}\right)-a.s.
\]
where $\widetilde{L}\varphi:=L\varphi+\frac{1}{2}\sum_{k=1}^{d}\Lambda_{k}\Lambda_{k}\varphi$.
\end{definition} 

\begin{remark} Above definition comes of course with a caveat: $L_{loc}^{1}$
is not a Banach space and the standard uniqueness results do not apply.
However, note that we could have given a more restrictive definition
of a weak solution by using weighted $L^{p}$ or Orlicz-spaces instead
of $L_{loc}^{1}$. Either way, we are not aware of a uniqueness theory
for degenerate SPDEs in either such a setup which seems to be a challenging
question in its own right. Below we only give the existence for $L_{loc}^{1}$-weak
solutions by showing that the viscosity RPDE solution is a $L_{loc}^{1}$-solution.\end{remark} 

\begin{proposition} Let $B$ be a $d$-dimensional Brownian motion,
$u_{0}\in BUC$ and assume $L$,$\Lambda$ fulfill the conditions
of theorem \ref{ThmMain}. Then $\left(u_{t}^{\boldsymbol{B}}\right)_{t\in\left[0,T\right]}$
is a $L_{loc}^{1}$-solution.\end{proposition} \begin{proof} For
$\epsilon>0$ consider the elliptic operator $L^{\epsilon}:=L+\epsilon\sum_{i=1}^{d}\partial_{i}^{2}$
and truncate and mollify $u_{0}$ to get $u_{0}^{\epsilon}\in C_{c}^{\infty}$
s.t.~$u_{0}^{\epsilon}\rightarrow u_{0}$ uniformly on compacts in
$\mathbb{R}^{n}$. Proposition \ref{prop:weak L2} shows for $\epsilon>0$
that the RPDE solution $u^{\epsilon}\in BUC\left(\left[0,T\right]\times\mathbb{R}^{n}\right)$
associated with $\left(L^{\epsilon},\Lambda^{\epsilon},u_{0}^{\epsilon},\boldsymbol{B}\right)$
is (in the equivalence class of) the unique $L^{2}$-solution; especially
$u^{\epsilon}$ is a $L_{loc}^{1}$-solution and therefore fulfills
$\left(d\lambda\otimes d\mathbb{P}\right)-a.s.$ that 
\[
\left\langle u_{.}^{\epsilon},\varphi\right\rangle -\left\langle u_{0}^{\epsilon},\varphi\right\rangle =\int_{0}^{.}\left\langle u_{r}^{\epsilon},\left(\widetilde{L^{\epsilon}}\right)^{\star}\varphi\right\rangle dr+\epsilon\int_{0}^{.}\left\langle u_{r}^{\epsilon},\sum_{i=1}^{d}\partial_{i}^{2}\varphi\right\rangle dr+\sum_{k=1}^{d}\int_{0}^{.}\left\langle u_{r}^{\epsilon},\Lambda_{i}^{\star}\varphi\right\rangle dB_{r}^{k}
\]
Conclude by noting that the locally uniform converge $u^{\epsilon}\rightarrow u$
on $\left[0,T\right]\times\mathbb{R}^{n}$ follows from the stability
properties of standard viscosity solutions ($u^{\epsilon}$ is given
by a transformation with RDE flows as a standard viscosity solution
with an extra term including a Hessian which vanishes as $\epsilon\rightarrow0$).
\end{proof}

\section{Applications to stochastic partial differential equations\label{AppSPDE}}

We now discuss some further applications of theorem \ref{ThmMain}
when applied to a stochastic driving signal, i.e.\ by taking $\mathbf{z}$
to be a realization of a continuous semi-martingale $Z$ and its stochastic
area, say $\mathbf{Z}\left(\omega\right)=\left(Z,A\right)$; the most
prominent example being Brownian motion and Lévy's area.

\begin{remark}{[}Itô versus Stratonovich{]} Note that similar \textbf{SPDEs
in Itô-form} need not be, in general, well-posed. Consider the following
(well-known) linear example 
\[
du=u_{x}dB+\lambda u_{xx}dt,\,\,\lambda\geq0.
\]
A simple computation shows that $v\left(x,t\right):=u\left(x-B_{t},t\right)$
solves the (deterministic) PDE $\dot{v}=\left(\lambda-1/2\right)v_{xx}$.
From elementary facts about the heat equation we recognize that, for
$\lambda<1/2$, this equation, with given initial data $v_{0}=u_{0}$,
is not well-posed. In the (Itô-)\ SPDE literature, starting with
\cite{MR553909}, this has led to coercivity conditions, also known
as super-parabolicity assumptions, in order to guarantee well-posedness.
\end{remark}

\begin{remark}{[}Regularity of noise coefficients{]} Applied in the
semimartingale context (finite $p$-variation for any $p>2$) the
regularity assumption of theorem \ref{ThmMain} reads Lip$^{4+\varepsilon}$,
$\varepsilon>0$. While our arguments do not appear to leave much
room for improvement we insist that working directly with Stratonovich
flows (rather than rough flows) will not bring much gain: the regularity
requirements are essentially the same. Itô flows, on the other hand,
require one degree less in regularity. In turn, there is a potential
loss of well-posedness and the resulting SPDE is not robust as a function
of its driving noise (similar to classical Itô stochastic differential
equations). \end{remark}

\begin{remark}{[}Space-time regularity of SPDE\ solutions{]} Since
$u\left(t,x\right)$ is the image of a (classical) PDE solution under
various (inner and outer) flows of diffeomorphisms, it suffices to
impose conditions on the coefficients on $L$ which guarantee that
existence of nice solutions to $\partial_{t}+\,\left[^{\phi^{\mathbf{z}}}\left(L^{\psi^{\mathbf{z}}}\right)\right]_{\alpha^{\mathbf{z}}}$.
For instance, if the driving rough path $\mathbf{z}$ has $1/p$-Hölder
regularity, it is not hard to formulate conditions that guarantee
that the rough PDE solutions is an element of $C^{1/p,2+\delta}$
for suitable $\delta>0$. Indeed, it is sheer matter of conditions-book-keeping
to solve $\partial_{t}+\,\left[^{\phi^{\mathbf{z}}}\left(L^{\psi^{\mathbf{z}}}\right)\right]_{\alpha^{\mathbf{z}}}$
under (known and sharp) conditions in Hölder spaces, cf.\ \cite[Section 9, p. 140]{KrylovHoelderBook},
with $C^{1+\delta/2,2+\delta}$ regularity. Unwrapping the change
of variables will not destroy spatial regularity (thanks to sufficient
smoothness of our diffeomorphisms for fixed $t$) but will most definitely
reduce time regularity to $1/p$-Hölder. \end{remark}

We now turn to the applications. Throughout we prefer to explain the
main point rather than spelling out theorems under obvious conditions;
the reader with familiarity with rough path theory will realize that
formulating and proving such statements follows easily from well-known
results once continuous dependence in rough path norm is established
(which is done in Theorem \ref{ThmMain}).

\textbf{(Approximations)} \textit{Any} approximation result to a Brownian
motion $B$ (or more generally, a continuous semimartingale) in rough
path topology implies a corresponding (weak or strong) limit theorem
for such SPDEs: it suffices that an approximation to $B$ converges
in rough path topology; as is well known (e.g.\ \cite[Chapter 13]{friz-victoir-book}
and the references therein) examples include piecewise linear, mollifier,
and Karhunen-Loeve approximations, as well as (weak) Donsker type
random walk approximations \cite{BFH08}. The point being made, we
shall not spell out more details here.

\textbf{(Support results)} In conjunction with known support properties
of $\mathbf{B}$ (e.g.\ \cite{LeQiZh02} in $p$-variation rough
path topology or \cite{friz-lyons-stroock-06} for a conditional statement
in Hölder rough path topology) continuity of the SPDE solution as
a function of $\mathbf{B}$ immediately implies Stroock--Varadhan
type support descriptions for such SPDEs. In the linear, Brownian
noise case, approximations and support of SPDEs have been studied
in great detail \cite{MR1140746,MR1026781,MR1019596,MR1011658,MR1008230}.

\textbf{(Large deviation results)} Another application of our continuity
result\ is the ability to obtain large deviation estimates when $B$
is replaced by $\varepsilon B$ with $\varepsilon\rightarrow0$; indeed,
given the known large deviation behaviour of $\left(\varepsilon B,\varepsilon^{2}A\right)$
in rough path topology (e.g.\ \cite{LeQiZh02} in $p$-variation
and\ \cite{friz-victoir-05} in Hölder rough path topology) it suffices
to recall that large deviation principles are stable under continuous
maps.

\textbf{(SPDEs with non-Brownian noise)} Yet another benefit of our
approach is the ability to deal with SPDEs with non-Brownian and even
non-semimartingale noise. For instance, one can take $\mathbf{z}$
as (the rough path lift of) fractional Brownian motion with Hurst
parameter $1/4<H<1/2\,$, cf. \cite{coutin-qian-02} or \cite{friz-victoir-2007-gauss},
a regime which is ``rougher\textquotedblright{}\ than Brownian and
notoriously difficult to handle, or a diffusion with uniformly elliptic
generator in divergence form with measurable coefficients; see \cite{friz-victoir-subelliptic}.
Much of the above (approximations, support, large deviation) results
also extend, as is clear from the respective results in the above-cited
literature.

\textbf{(SPDEs with higher level rough paths without extra effort)}
In contrast to the approach by Gubinelli-Tindel \cite{gubinelli-tindel-2008},
no extra effort is necessary when $\mathbf{z}$ is a higher level
rough path (the prominent example being fractional Brownian motion
with $H\in(1/4,1/3]$).

\textbf{(Approximation of Wong-Zakai type with modified drift term)}
For brevity let us write $L,$ $\Lambda$ and $\Lambda_{k}$ instead
of $L(t,x,u,Du,D^{2}u)$, $\Lambda(t,x,u,Du)$ and $\Lambda_{k}(t,x,u,Du)$
in this section and consider the SPDE
\[
du+Ldt=\sum_{k}\Lambda_{k}\circ dZ^{k}\text{.}
\]
Equivalently, we write
\[
du+Ldt=\Lambda d\mathbf{Z}
\]
where $\mathbf{Z}$ denotes the Stratonovich lift of $\left(Z^{1},\dots,Z^{d}\right)$.
Recall that $\log\mathbf{Z}$ takes values in $\mathbb{R}^{d}\oplus so\left(d\right)$.
Define $\mathbf{\tilde{Z}}$ by perturbing the Lévy area as follows
\[
\log\mathbf{\tilde{Z}}:=\log\mathbf{Z}+\left(0,\Gamma_{t}\right)
\]
where $\Gamma\in C^{1\text{-var}}\left(\left[0,T\right],so\left(d\right)\right)$.
Then the solution to
\[
d\tilde{u}+Ldt=\Lambda d\mathbf{\tilde{Z}}
\]
is identified with
\[
d\tilde{u}+Ldt=\Lambda d\mathbf{\tilde{Z}+}\sum_{i,j\in\left\{ 1,\ldots,d\right\} }\left[\Lambda_{i},\Lambda_{j}\right]d\Gamma^{i,j}.
\]
The practical relevance is that one can construct approximations $\left(Z^{n}\right)$
to $Z$, each $Z^{n}$ only dependent on finitely many points, which
converge uniformly to $Z$ with the ``wrong'' area (cf.\ \cite{friz-oberhauser-2008b});
that is,
\[
\left(Z^{n},\int Z^{n}dZ^{n}\right)\rightarrow\mathbf{\tilde{Z}}
\]
in $p$-variation rough path sense, $p\in\left(2,3\right)$. The solutions
to the resulting approximations will then converge to the solution
of the ``wrong'' limiting equation
\[
d\tilde{u}+Ldt=\sum_{k=1}^{d}\Lambda_{k}\circ dZ^{k}\mathbf{+}\sum_{i,j\in\left\{ 1,\ldots,d\right\} }\left[\Lambda_{i},\Lambda_{j}\right]d\Gamma^{i,j}.
\]
The formal proof is easy; it suffices to analyze the equations (rough)
differential equations for $\left(\psi,\phi,\alpha\right)$ in presence
of area perturbation; see \cite{friz-oberhauser-2008b}, and then
identify the corresponding operators $\left[^{\phi}\left[L^{\psi}\right]\right]_{\alpha}$.
Actually, one can pick any multi-index $\gamma=\left(\gamma_{1},\ldots,\gamma_{N}\right)\in\left\{ 1,\ldots,d\right\} ^{N}$
and find (uniform) approximations such as to make the higher iterated
Lie brackets $\Lambda_{\gamma}=\left[\Lambda_{\gamma_{1}},\ldots,\left[\Lambda_{\gamma_{N-1}},\Lambda_{\gamma_{N}}\right]\ldots\right]$,
or even any linear combination of them, appear by perturbing the higher
order iterated integrals.\\
\textbf{(SPDEs with delayed Brownian input)} A interesting concrete
example of the previous discussion arises when the $2$-dimensional
driving signal is Brownian motion with its $\varepsilon$-delay; say
\[
du^{\varepsilon}+Ldt=\Lambda_{1}\circ dB_{\cdot-\varepsilon}^{\varepsilon}+\Lambda_{2}\circ dB_{\cdot}
\]
where $B_{t-\varepsilon}^{\varepsilon}:=B_{t-\varepsilon}$. Observe
that in the classical setting this can be solved (as flow) on $\left[0,\varepsilon\right]$,
then on $\left[\varepsilon,2\varepsilon\right]$ and so on. As $\varepsilon\rightarrow0,$
$(B_{t}^{\varepsilon},B_{t})$ converges in rough path sense to $\left(B_{t},B_{t}\right)$
with non-trivial area $-t/2$ (see \cite[Chapter 14]{friz-victoir-book}).
In other words, $u^{\varepsilon}\rightarrow u$ in probability and
locally uniformly where
\[
du+Ldt=\left(\Lambda_{1}+\Lambda_{2}\right)\circ dB+\left[\Lambda_{1},\Lambda_{2}\right]dt
\]
\textbf{(Robustness of the Zakai SPDE in nonlinear filtering)}\textbf{\noun{
}}Nonlinear filtering is concerned with the estimation of the conditional
law of a Markov process; to be precise consider 
\begin{align}
dX_{t} & =\mu\left(X_{t}\right)dt+V\left(X_{t}\right)dB_{t}+\sigma\left(X_{t}\right)d\tilde{B_{t}}\label{eq:filter}\\
dY_{t} & =h\left(X_{t}\right)dt+d\tilde{B}_{t}\nonumber 
\end{align}
where $B$ and $\tilde{B}$ are independent, multidimensional Brownian
motions. The goal is to compute for a given real-valued function $\varphi$
\[
\pi_{t}\left(\varphi\right)=\mathbb{E}\left[\varphi\left(X_{t}\right)\lvert\sigma\left(Y_{r},r\leq t\right)\right]
\]
and from basic principles it follows that there exists a map $\phi_{t}^{\varphi}:C\left(\left[0,T\right],\mathbb{R}^{d_{Y}}\right)\rightarrow\mathbb{R}$
such that 
\begin{equation}
\phi_{t}^{\varphi}\left(Y\lvert_{\left[0,t\right]}\right)=\pi_{t}\left(\varphi\right)\,\,\,\mathbb{P}-a.s.\label{eq:robust}
\end{equation}
As pointed out by Clark \cite{ClarkRobustness}, this classic formulation
is not justified in practice since only discrete observations of $Y$
are available and the functional $\phi_{t}^{\varphi}$ is only defined
up to nullsets on pathspace (which includes the in practice observed,
bounded variation path). He showed that in the case of uncorrelated
noise ($\sigma\equiv0$ in \eqref{eq:filter}) there exists a unique
$\overline{\phi}_{t}^{\varphi}:C\left(\left[0,T\right],\mathbb{R}^{n}\right)\rightarrow\mathbb{R}$
which is continuous in uniform norm and fulfills \eqref{eq:robust},
thus providing a version of the conditional expecation $\pi_{t}\left(\varphi\right)$
which is robust under approximations in uniform norm of the observation
$Y$. Unfortunately in the correlated noise case this is no longer
true!%
\footnote{We quote Mark Davis \cite{DavisRobustness}\begin{quote}``It must,
regretfully, be pointed out that the results for correlated noise
cannot, unlike those for the independence case, be extended to vector
observations. This is because the corresponding operators (...) do
not in general commute whereas with no noise correlation they are
multiplication operators which automatically commute.'' \end{quote}See
also the counterexample given in \cite{CrisanDiehlFrizOberhauser}.%
} In \cite{CrisanDiehlFrizOberhauser} it was recently shown that in
this case robustness still holds in a rough path sense. Now recall
that under well-known conditions \cite{MR553909,MR1135324,bain2008fundamentals},
$\pi_{t}$ can be written in the form
\begin{equation}
\pi_{t}\left(\varphi\right)=\int_{\mathbb{R}^{d_{X}}}\varphi\left(x\right)\frac{u_{t}\left(x\right)}{\int u_{t}\left(\tilde{x}\right)d\tilde{x}}dx\label{eq:pi}
\end{equation}
where $u_{t}\in L^{1}\left(\mathbb{R}^{n}\right)$ a.s.~and $\left(u_{t}\right)$
is the $L^{2}$-solution of the (dual) Zakai SPDE 
\begin{align}
du_{t} & =G^{\star}dt+\sum_{k}N_{k}u_{t}dY_{t}^{k}\nonumber \\
 & =\left(G^{\star}+\frac{1}{2}\sum_{k}N_{k}N_{k}\right)u_{t}dt+\sum_{k}N_{k}u_{t}\circ dY_{t}^{k}\label{eq:zakai}
\end{align}
with $G$ denoting the generator of the diffusion $X$, $Y$ a Brownian
motion under a measure change and 
\begin{equation}
\left(N_{k}u\right)\left(t,x\right)=\sigma_{k}^{i}\left(t,x\right)\partial_{i}u_{t}\left(x\right)+h\left(x\right).u_{t}\left(x\right).\label{eq:gradien tnoise}
\end{equation}
Using Theorem \eqref{ThmMain} in combination with Proposition \ref{prop:weak L2}
one can now construct a solution of \eqref{eq:zakai} which depends
continuously on the observation in rough path metric. The results
in \cite{CrisanDiehlFrizOberhauser} (where one works directly with
Kallianpur\textendash{}Striebel functional) suggest that one can use
the representation \ref{eq:pi} to establish robustness. However,
to this end it is necessary to show that $u_{t}^{\boldsymbol{z}}\in L^{1}$
(i.e.~a rough pathwise version of the discussion in \cite[Chapter 5]{MR1135324})
which we hope to discuss in the future in detail. Finally, let us
note that the gradient term in the noise $N_{k}u$ explains rather
intuitively why in the general, correlated noise case rough path metrics
are required: as pointed out above, correction terms are picked up
by the brackets $\left[N_{i},N_{j}\right]$ but if $\sigma=0$ then
$\left[N_{i},N_{j}\right]=0$, hence no extra terms appear. In fact,
solving \eqref{eq:zakai} for the case of $\sigma\equiv0$ reduces
via the method of characteristics to solving an SDE with commuting
vector fields which is well-known to be robust under approximations
of the driving signal (i.e.~the observation $Y$) in uniform norm.

\section{Appendix: comparison for parabolic equations}

Let $u\in\mathrm{BUC}\left([0,T]\times\mathbb{R}^{n}\right)$ be a
subsolution to $\partial_{t}+F$; that is, 
\[
\partial_{t}u+F\left(t,x,u,Du,D^{2}u\right)\leq0
\]
 if $u$ is smooth and with the usual viscosity definition otherwise.
Similarly, let $v\in\mathrm{BUC}\left([0,T]\times\mathbb{R}^{n}\right)$
be a supersolution.

\begin{theorem}\label{thm:comp}Assume condition (3.14) of the User's
Guide \cite{MR1118699UserGuide}, uniformly in $t$, together with
uniform continuity of $F=F\left(t,x,r,p,X\right)$ whenever $r,p,X$
remain bounded. Assume also a (weak form of) properness: there exists
$C$ such that
\begin{equation}
F\left(t,x,r,p,X\right)-F\left(t,x,s,p,X\right)\geq C\left(r-s\right)\,\,\,\forall r\leq s,\label{WeakProper}
\end{equation}
and for all $t\in\left[0,T\right]$ and all $\, x,p,X$. Then comparison
holds. That is,
\[
u\left(0,\cdot\right)-v\left(0,\cdot\right)\implies u\leq v\text{ on }[0,T]\times\mathbb{R}^{n}.
\]
\end{theorem}

\begin{proof} It is easy to see that $\tilde{u}=e^{-Ct}u$ is a subsolution
to a problem which is proper in the usual sense; that is (\ref{WeakProper})
holds with $C=0$ which is tantamount to require that $F$ is non-decreasing
in $r$. The standard arguments (e.g. \cite{MR1118699UserGuide} or
the appendix of \cite{MR2765508} or also \cite{DiehlFrizOberhauser})
then apply with minimal adaptations. \end{proof}

\begin{corollary} Under the assumptions of the theorem above let
$u,v$ be two solutions, with initial data $u_{0},v_{0}$ respectively.
Then 
\[
\left\vert u-v\right\vert _{\infty;\mathbb{R}^{n}\times\left[0,T\right]}\leq e^{CT}\left\vert u_{0}-v_{0}\right\vert _{\infty;\mathbb{R}^{n}}
\]
with $C$ being the constant from $\left(\ref{WeakProper}\right)$.
\end{corollary}

\begin{proof} Use again the transformation $\tilde{u}=e^{-Ct}u,\tilde{v}=e^{-Ct}v$.
Then $\tilde{v}+\left\vert u_{0}-v_{0}\right\vert _{\infty;\mathbb{R}^{n}}$
is a supersolution of a problem to which standard comparison arguments
apply; hence, 
\[
\tilde{u}\leq\tilde{v}+\left\vert u_{0}-v_{0}\right\vert _{\infty;\mathbb{R}^{n}}.
\]
Applying the same reasoning to $\tilde{u}+\left\vert u_{0}-v_{0}\right\vert _{\infty;\mathbb{R}^{n}}$
and finally undoing the transformation gives the result. \end{proof}

\bibliographystyle{plain}
\bibliography{/homes/stoch/oberhaus/Dropbox/projects/BibteX/roughpaths,/home/hd/Dropbox/projects/BibteX/roughpaths}

\end{document}